\documentclass[12pt]{amsart}
\usepackage[T1]{fontenc}
\usepackage[pdftex]{color,graphicx}
\usepackage{url}
\usepackage{mathtools}
\mathtoolsset{showonlyrefs}
\usepackage{anysize}
\usepackage{geometry}
\marginsize{2.5cm}{2.5cm}{2.5cm}{2.5cm}
\usepackage{amssymb}
\usepackage{hyperref}
\usepackage{amsaddr}
\theoremstyle{plain}
\numberwithin{equation}{section}
\newtheorem{theorem}{Theorem}[section]
\newtheorem{lemma}[theorem]{Lemma}
\newtheorem{corollary}[theorem]{Corollary}

\newtheorem{remark}[theorem]{Remark}

\newcommand{\ii}{\mathsf{i}}

\begin{document}
\title[A complete characterization of Schur stability for trinomials]{The stability region for Schur stable trinomials with general complex coefficients}

\author{Gerardo Barrera}
\address{University of Helsinki, Department of Mathematics and Statistics.
P.O. Box 68, Pietari Kalmin katu 5, FI-00014. Helsinki, Finland.\\
\url{gerardo.barreravargas@helsinki.fi}\\
\url{https://orcid.org/0000-0002-8012-2600}}
\thanks{*Corresponding author: Gerardo Barrera.}
\author{Waldemar Barrera}
\address{Facultad de Matem\'aticas, Universidad Aut\'onoma de Yucat\'an. Anillo Perif\'erico Norte Tablaje CAT 13615, M\'erida, Yucat\'an, M\'exico.\\
\url{bvargas@correo.uady.mx}\\
\url{https://orcid.org/0000-0001-6885-5556}}
\author{Juan Pablo Navarrete}
\address{Facultad de Matem\'aticas, Universidad Aut\'onoma de Yucat\'an. Anillo Perif\'erico Norte Tablaje CAT 13615, M\'erida, Yucat\'an, M\'exico.\\
\url{jp.navarrete@correo.uady.mx}\\
\url{https://orcid.org/0000-0002-3930-4365}}


\begin{abstract}
In this paper, we characterize the stability region for trinomials of the form $f(\zeta):=a\zeta ^n + b\zeta ^m +c$, $\zeta\in \mathbb{C}$, where $a$, $b$ and $c$ are non-zero complex numbers and $n,m\in \mathbb{N}$ with $n>m$.
More precisely, we provide necessary and sufficient conditions on the coefficients $a$, $b$ and $c$ in order that all the roots of the trinomial $f$ belongs to the open unit disc in the complex plane. The proof is based on Bohl's Theorem~\cite{Bohl} introduced in 1908.
\end{abstract}
\maketitle
\tableofcontents

\section{\textbf{Introduction}}
\subsection{\textbf{The stability problem}}
The computation and the quantitative location of the roots of a given polynomial are ubiquitous in the applications and it has been produced a vast literature in mathematics and in applied mathematics in recent years.
It is well-known that a linear discrete dynamical system,
for instance linear recurrence equations which are widely used in applied mathematics
and computer science for modeling the future of a process
that depends linearly on a finite string, is asymptotically stable if and only if its corresponding characteristic polynomial has all its roots with complex modulus strictly smaller than one. When a polynomial has all its roots with complex modulus strictly less than one,  the polynomial is called a \textit{Schur stable polynomial}. The notion of Schur stable polynomials is originated in the study of stability of dynamical systems, in particular, in the so-called control theory.

Recall that a general linear recurrence equation with constant coefficients and two delays is defined as follows. 
For a given initial string of complex numbers $\varphi_0,\varphi_1,\ldots,\varphi_{n-1}$, let $(\varphi(t))_{t\geq n}$ be the unique solution of the discrete-time initial value problem
\begin{equation}\label{eq:nuevp}
\left\{
\begin{array}{r@{\;=\;}l}
X(t)&-bX(t-(n-m))-cX(t-n)\quad \textrm{ for }\quad t\in \{n,n+1,n+2,\ldots,\},\\
X(t)&\varphi_t\quad \quad \textrm{ for }\quad t\in \{0,1,\ldots,n-1\},
\end{array}
\right.
\end{equation}
where $b$ and $c$ are non-zero complex numbers and $n,m\in \mathbb{N}$ with $n>m$. The time-shifts  $n-m$ and $n$ in~\eqref{eq:nuevp} are called delays.
It is well-known that the characteristic polynomial 
associated to~\eqref{eq:nuevp} is given by
$T(\zeta):= \zeta ^n + b\zeta ^m +c$, $\zeta\in \mathbb{C}$,
and then the dynamical system~\eqref{eq:nuevp} is asymptotically stable 
 if and only if $T$ is a Schur stable polynomial. For more details about the theory of linear recurrence equations, we refer to the monographies~\cite{Elaydi2005,Jerribook,Mickens}.
  
Discrete-time stable dynamical systems of the form~\eqref{eq:nuevp} with real coefficients have been broadly used in modeling,  
for instance, 
they have been used in Numerical Analysis for the numerical discretization of the so-called linear delay differential equation, in Mathematical Biology for the linearization process of various discrete population growth models, in Financial Mathematics to determine the interest rate, the amortization of a loan and price fluctuations, in Probability for the so-called first step analysis of Markov chains, in Statistics for
the autoregressive linear model,
see~\cite{Botta,Clark,DananElaydi2004,
Elaydi2005,LevinMay,Privault}.
While systems of the form~\eqref{eq:nuevp} with complex coefficients naturally appear in the study of systems of linear recurrence equations such as 
in the linearized model for the
discrete-time Hopfield network with a single delay, or in local stability analysis of some discrete-time population dynamics models, see \cite{Guo,Kaslik,Kipnis2011,Matsunaga2010}.

Since the stability for linear recurrence equations with constant coefficients and two delays is equivalent to the Schur stability of trinomials,
in this paper we parametrize the stability region for trinomials of the form
\begin{equation}\label{def:f}
f(\zeta):=a\zeta ^n + b\zeta ^m +c, \quad \zeta\in \mathbb{C},
\end{equation}
where $a$, $b$ and $c$ are complex numbers and $n,m\in \mathbb{N}$ with $n>m$.
In other words, we provide necessary and sufficient conditions on the complex coefficients $a$, $b$ and $c$ in order that all the roots of $f$ belongs to the open unit disc in the complex plane. 
In that case, we say that~\eqref{def:f} is a \textit{Schur stable trinomial}.
The latter is straightforward when some of the coefficients $a$, $b$ or $c$ are zero. Indeed, we observe that
\begin{equation}\label{eq:degenerados}
\textrm{ $f$ is a Schur stable trinomial if and only if }
\begin{cases}
|c|/|b|<1 &
\textrm{ for $a=0$ and $b\neq 0$,}\\
|c|<1 & \textrm{ for $a=0$ and $b=0$,} \\
|c|/|a|<1 & \textrm{ for $a\neq 0$ and $b=0$,}\\
|b|/|a|<1 & \textrm{ for $a\neq 0$, $b\neq 0$ and $c=0$,}  
\end{cases}
\end{equation}
where $|\cdot|$ denotes the complex modulus.
Therefore, without loss of generality, we always assume that $a$, $b$ and $c$ are non-zero complex numbers.

The celebrated works of P. Ruffini, N. H. Abel and \'E. Galois yield that for $n\geq 5$ generically there is no formula for the roots of~\eqref{def:f} in terms of radicals. We recommend~\cite{Borweinbook,Mardenbook,Prasolovbook} for treatises on polynomials and~\cite{Pan1997} for a brief history of solving polynomials.
Trinomials of the form~\eqref{def:f} with real coefficients have been the subject of numerous qualitative and quantitative studies because of their theoretical importance as well as their applications, see~\cite{Belki} and the list of references therein.
There is a vast literature reporting the study of location of the roots of trinomials when its coefficients are real numbers including series representations of the roots and the shape of the stability region,~\cite{AhnKim,Aziz,Botta,
BrilleslyperSchaubroeck2018,BrilleslyperSchaubroeck2014,
Cermak2015survey,CermakFedorkova2022,Cermak2022,
Cermak2015,Cermak2019,Cheng2009,Danan2004,Darling1924,
Dilcher1992,Eagle,eger,Geleta,Hall,Hanov,Hernane,
Howell,Jain,Kennedy,Kipnis2004,Kuruklis,Marques,
Matsunaga2010,Nekrassoff,Nulton,Papanicolaou1996,
Ren2007,Scoupas,Uahabi,Wang}. In the case of trinomials with complex coefficients, 
lower and uppers bounds for the moduli of their roots have been also obtained 
in~\cite{Cermak2019,daSilvaSri2005,Kennedy,Matsunaga2010,Otto,Vassilev}. 
Trinomials have been also studied from the geometrical, topological and dynamical perspectives, see for 
instance~\cite{Aguirre,BarreraCano,BarreraMaganaNavarrete,Bohl,Fell,Ferrerbook,JimenezMunoz,Kasten,Melman,MunozSeoane,Szabo,Theobald} and the references therein.
Using the Cohn reduction degree method (see~\cite{Cohn1922} or Lemma~42.1 in~\cite{Mardenbook}), the stability region for trinomials with real coefficients is  given in~\cite{Cermak2015}.
Recently, the authors in~\cite{Cermak2022} apply Bolh's Theorem (see Theorem~\ref{th:bohl} below) for trinomials with real coefficients and obtain the results in~\cite{Cermak2015}. 
Nevertheless, to the best of our knowledge, using Bohl's Theorem for the case with complex coefficients has not been fully characterized and it does not follow  straightforwardly from the real case, see Subsection~\ref{subsec:preli} for an explanation of the difficulties on the counting argument in the complex case. 

\subsection{\textbf{Preliminaries}}\label{subsec:preli}
Along this manuscript,  $n>m>0$ are fixed. 
Our main tool is Bohl's Theorem given in~\cite{Bohl}.
Bohl's Theorem gives the number of roots of~\eqref{def:f}
in an open ball of radius $r$ centered at the origin 
 according to whether the non-negative
numbers $|a|r^n$, $|b|r^m$
and $|c|$ are the lengths  of  the sides of some triangle (including degenerate triangles), or not. Bohl's Theorem reads as follows:

\begin{theorem}[Bohl's Theorem for trinomials~\cite{Bohl,Cermak2022}]\label{th:bohl}
\hfill

\noindent
Assume that $n$ and $m$ are coprime numbers.
Let $r>0$ and assume that $|a|r^n$, $|b|r^m$
and $|c|$ are the side lengths of some triangle (it may be degenerate
except for the case below). Let $\omega_1$ and $\omega_2$ be the opposite angles to the sides with lengths $|a|r^n$ and $|b|r^m$, respectively. 
Then the number of roots of~\eqref{def:f} in the open disc of radius $r$, $D_r:=\{z\in \mathbb{C}:|z|<r\}$, is equal to the number of integers in the open interval 
$(P-\omega(r),P+\omega(r))$,
where
\begin{equation}\label{def:PW}
P:=\frac{n(\beta-\gamma+\pi)-m(\alpha-\gamma+\pi)}{2\pi},\quad \quad
\omega(r):=\frac{n\omega_1+m\omega_2}{2\pi},
\end{equation}
and $\alpha$, $\beta$, $\gamma$ are the arguments of $a$, $b$, $c$, respectively.
However, when $|b|r^m=|a|r^n+|c|$, $r^{n-m}>\frac{m |b|}{n|a|}$ and $P+\omega(r)$ is an integer, the number of roots of~\eqref{def:f} in $D_r$ is equal to $m$.

Moreover, when $|a|r^n$, $|b|r^m$
and $|c|$ are not the side lengths of any triangle, then 
\begin{equation}\label{eq:bolhsegunda}
\textrm{The number of roots of $f$ given 
in~\eqref{def:f} in $D_r$}=
\begin{cases}
0 & \textrm{ if }\quad |c|>|a|r^n+|b|r^m,\\
m & \textrm{ if }\quad |b|r^m>|a|r^n+|c|,\\
n & \textrm{ if }\quad |a|r^n>|b|r^m+|c|.
\end{cases}
\end{equation}
\end{theorem}

\begin{remark}[Exceptional cases]
\hfill

\noindent
The original statement of Bohl's Theorem given in~\cite{Bohl} includes two additional exceptions, see Item~(a) and Item~(b) in Theorem~\ref{th:Bohloriginal} in Appendix~\ref{ap:Bohlstatement}.
Nevertheless, when $n$ and $m$ are coprime numbers, such exceptions have been already included in Theorem~\ref{th:bohl} in the case when the triangle is degenerate. We discuss it in fully detail in the 
Appendix~\ref{ap:Bohlstatement}.
\end{remark}

From now on to the end of this manuscript, we assume that $a=1$ and  in concious abuse of notation we write
\begin{equation}\label{def:fa}
f(\zeta):=\zeta ^n + b\zeta ^m +c, \quad \zeta\in \mathbb{C}.
\end{equation}
We start with the following observation. If $\textsf{gcd}(n,m)=\ell\in \{2,3,\ldots,m\}$, where $\textsf{gcd}$ denotes the greatest common divisor function, we set
$\widetilde{n}:=n/\ell$ and $\widetilde{m}:=m/\ell$, which satisfy 
$\textsf{gcd}(\widetilde{n},\widetilde{m})=1$. Then the change of variable $\zeta\mapsto \zeta^\ell$ yields that 
$f$ is a Schur stable trinomial if and only if 
\begin{equation}
\widetilde{f}(\zeta):=\zeta^{\widetilde{n}} + b\zeta^{\widetilde{m}} +c, \quad \zeta\in \mathbb{C}\quad \textrm{ is a Schur stable trinomial.}
\end{equation}
Therefore, without loss of generality, from here to the end of the manuscript, we assume that  $\textsf{gcd}(n,m)=1$. In addition, 
if~\eqref{def:fa} is a Schur stable trinomial, then the celebrated Vi\`ete's formulas  yield that $|c|<1$.

For any $s\in \mathbb{R}$ we define the \textit{angular flow} $g_s:\mathbb{C}\to \mathbb{C}$ by
\begin{equation}\label{eq:angular}
\begin{split}
g_s(\zeta):&=e^{-\ii ns}f(e^{\ii s}\zeta)\\
&=e^{-\ii ns}(e^{\ii ns}\zeta^n+be^{\ii ms}\zeta^m+c)\\
&=\zeta^n+be^{-\ii (n-m)s}\zeta^m+e^{-\ii ns}c,
\end{split}
\end{equation}
where $\ii$ denotes the unit imaginary number. We point out that the set of roots of $g_s$ are exactly the set of roots of $f$ up to a rotation $e^{-\ii s}$. In other words, for any $s\in \mathbb{R}$, $f$ is a Schur stable trinomial if and only if $g_s$ is a Schur stable trinomial.

The Fundamental Theorem of Algebra yields that
two trinomials $f$ and $g$  have the same roots if and only if $f=\lambda g$ for some $\lambda\in \mathbb{C}^*:=\mathbb{C}\setminus\{0\}$. The preceding equivalence is in fact an equivalence relation. Then
it is natural to identify the subspace of trinomials (modulo the preceding equivalence relation) with the complex projective plane $\mathbb{P}^2_{\mathbb{C}}$.
We recall that
\begin{equation}
\mathbb{P}^2_{\mathbb{K}}:=\{[a:b:c]:\quad\textrm{ where }\quad [a:b:c]:=\{\lambda(a,b,c):\lambda\in \mathbb{K}^*\} \}
\end{equation}
for $\mathbb{K}$ being the field $\mathbb{R}$ or $\mathbb{C}$ and $\mathbb{K}^*:=\mathbb{K}\setminus\{0\}$.
Let 
\begin{equation}\label{def:omega}
\begin{split}
\Omega_n:&=\{[a:b:c]\in \mathbb{P}^2_{\mathbb{C}}:\, a\zeta^n+b\zeta^m+c\,\,\textrm{ is a Schur stable trinomial}\}\\
&=\{[1:b:c]\in \mathbb{P}^2_{\mathbb{C}}:\, \zeta^n+b\zeta^m+c\,\,\textrm{ is a Schur stable trinomial}\}
\end{split}
\end{equation}
and define the continuous projection
$\Pi_n: \Omega_n \to \mathbb{P}^2 _{\mathbb{R}}$
by
\begin{equation}\label{def:pi} 
\Pi_n([1:b:c])=[1:|b|:(-1)^n|c|].
\end{equation}
Then we naturally say that an element $[a:b:c]\in \mathbb{P}^2_{\mathbb{K}}$ is Schur stable over $\mathbb{K}$ if and only if $[a:b:c]\in \Omega_n\cap \mathbb{P}^2 _{\mathbb{K}}$.
In addition, the set $\Omega_n$ is called the stability region.
By Proposition~7.9 in~\cite{BarreraMaganaNavarrete} we have that $\Pi_n(\Omega_n)\subset \Omega_n\cap \mathbb{P}^2 _{\mathbb{R}}$, that is to say, the function $\Pi_n$ maps Schur stable trinomials with complex coefficients to Schur stable trinomials with real coefficients.
In addition, we remark that
the image of $\Pi_n$, i.e., $\Pi_n(\Omega_n)$, can be decomposed into two disjoint sets $``\Delta"$ and $``\Gamma"$ according to whether the
numbers $1$, $|b|$
and $|c|$ are the lengths  of  the sides of some triangle (it may be degenerate) or not, see 
Lemma~\ref{lem:piezas} in Appendix~\ref{ap:tools}.

In the recent paper~\cite{Cermak2022}, it is applied Theorem~\ref{th:bohl} to obtain a characterization of Schur stability 
of~\eqref{def:f} when its coefficients are real numbers. Nevertheless, when the coefficients are complex numbers such characterization is not straightforward. 
In fact, when the coefficients are complex numbers 
we note that the pivot $P$ takes continuous values while when the coefficients are real numbers the pivot $P\in \{\ell/2: \ell\in \mathbb{Z}\}$, see~\eqref{eq:Pcasos} below. 
For a fixed $\omega>0$ and $P$ taking continuous values, one can note that the number of integers in the open interval $(P-\omega,P+\omega)$ may take three different values.

In what follows we explain the main idea in order to apply Theorem~\ref{th:bohl} for the case of complex coefficients.
By Theorem~\ref{th:bohl}, the pivot $P$ associated to~\eqref{def:fa}
is given by 
\begin{equation}\label{eP}
P=\frac{n(\pi+\arg(b)-\arg(c))-m(\pi-\arg(c))}{2\pi}.
\end{equation}
In particular, for the trinomial associated to a point of the form
\[
[1:x:y] \quad \textrm{ with }\quad x\in \mathbb{R}\setminus\{0\}\quad \textrm{ and }\quad y \in \mathbb{R}\setminus\{0\},
\]
all the possible pivots $P$ are given by 
\begin{equation}\label{eq:Pcasos}
P=
\begin{cases}
\frac{n-m}{2} & \textrm{ for }\quad x>0\quad \textrm{ and }\quad y>0,\quad \textrm{i.e., first quadrant},\\
n-\frac{m}{2} & \textrm{ for }\quad x<0\quad \textrm{ and }\quad y>0,
\quad \textrm{i.e., second quadrant},\\
\frac{n}{2} & \textrm{ for }\quad x<0\quad \textrm{ and }\quad y<0,
\quad \textrm{i.e., third quadrant},\\
0 & \textrm{ for }\quad x>0\quad \textrm{ and }\quad y<0,
\quad \textrm{i.e., fourth quadrant}.
\end{cases}
\end{equation}
We observe that $P\in \mathbb{Z}\cup\{\ell+1/2:\ell\in \mathbb{Z}\}$. 
Consider the trinomial associated to the point $[1:b:c] \in \Omega_n$, where $b,c \in \mathbb{C}\setminus\{0\}$ with corresponding pivots $\widetilde{P}$ and $\widetilde{\omega}$ defined 
in~\eqref{def:PW} for $r=1$.
We note that the corresponding pivots (defined 
in~\eqref{def:PW} for $r=1$) for the trinomials associated to the points of the form 
\[
[1:\pm |b|:\pm |c|],
\]
are $P$ and $\omega=\widetilde{\omega}$, where $P$ satisfies~\eqref{eq:Pcasos}.
We stress that the intervals $\widetilde{I}:=(\widetilde{P}-\widetilde{\omega},\widetilde{P}+\widetilde{\omega})$ and $I:=(P-\omega,P+\omega)$ have the same length $2\widetilde{\omega}=2\omega$. However, the may not have the same number of integers on them. 
The cunning choice $\Pi_n$ defined 
in~\eqref{def:pi}
implies that $\widetilde{I}$ and $I$ have the same number of integers,
which allows us to reduce the problem to the real case,
see Proposition~7.9 in~\cite{BarreraMaganaNavarrete}.

\subsection{\textbf{Main results and its consequences}}
In this subsection, we state
the main result of this manuscript and its consequences.

The main result of this manuscript is the following characterization of $\Omega_n$.
\begin{theorem}[Schur stability for trinomials: projective notation]\label{th:main}
\hfill

\noindent
Assume that  $\textsf{gcd}(n,m)=1$.
Let $\Omega_n$ be the set defined in~\eqref{def:omega} and $\Pi_n$ be the projection given in~\eqref{def:pi}.
Then the following is valid.
Every element in $\Omega_n$  can be parametrized in the form
\[
[1: xe^{\ii t} e^{-\ii(n-m) s} : ye^{-\ii ns}],
\]
with
$[1:x:y] \in \Pi_n(\Omega_n)$, 
$0 \leq s\leq 2 \pi$  
and $t$ is a real number satisfying
\begin{enumerate}
\item $|t| \leq \frac{\pi}{n}$ for $[1:x:y] \in \Gamma_{n\mod 2}$,
\item $|t|<\frac{\pi (2 \omega-n+1)}{n}$ for 
$[1:x:y] \in \Delta_{n\mod 2}$,
\end{enumerate}
where 
\begin{equation}\label{eq:defgammadelta}
\begin{split}
\Gamma_{0}&:=\{[1:u:v] \in \mathbb{P}_{\mathbb{R}}^2 \, \,  
| \, \,  0 \leq u,\, \,  0 \leq v, \, \,  u+v<1\},\\
\Delta_{0}&:=\{[1:u:v] \in \mathbb{P}_{\mathbb{R}}^2 \, \,  
| \, \,  0 < u,\, \,  0 < v, \, \, u+v \geq 1 , \, \, 2\omega(u,v) >n-1\},\\
\Gamma_{1}&:=\{[1:u:v] \in \mathbb{P}_{\mathbb{R}}^2 \, \,  
| \, \,  0 \leq u,\, \,  v \leq 0, \, \,  u-v<1\},\\
\Delta_{1}&:=\{[1:u:v] \in \mathbb{P}_{\mathbb{R}}^2 \, \,  
| \, \,  0 < u,\, \,  v < 0, \, \, 
u-v \geq 1 , \, \, 2\omega(u,v) >n-1\},
\end{split}
\end{equation}
and
\begin{equation}\label{eq:defomega1}
\omega(u,v):= \frac{n\, \arccos \left(\frac{u^2+v^2-1}{2u|v|}\right)+m\, \arccos \left( \frac{1-u^2+v^2}{2|v|}\right)}{2\pi}.
\end{equation}

Conversely, every point of this form belongs in $\Omega _n$.
\end{theorem}

In the sequel,
we reformulate Theorem~\ref{th:main} using polynomial notation, which is convenient for the study of~\eqref{eq:nuevp}.
\begin{theorem}[Schur stability for trinomials: polynomial notation]\label{th:mainrec}
\hfill

\noindent
Assume that  $\textsf{gcd}(n,m)=1$.
The following is valid. 
Every Schur stable trinomial of the form~\eqref{def:fa} can be parametrized in the form
\begin{equation}\label{eq:schur}
\zeta^n+xe^{\ii t} e^{-\ii(n-m) s}\zeta^m + ye^{-\ii ns}
\end{equation}
with
$\zeta^n+x\zeta^m+(-1)^n y$ being a Schur stable trinomial, 
$0 \leq s\leq 2 \pi$  
and $t$ is a real number satisfying
\begin{enumerate}
\item $|t| \leq \frac{\pi}{n}$ for $(x,y)\in \Gamma_{n\mod 2}$,
\item $|t|<\frac{\pi (2 \omega-n+1)}{n}$ for 
$(x,y)\in \Delta_{n\mod 2}$,
\end{enumerate}
where 
\begin{equation}\label{eq:defgammadeltarec}
\begin{split}
\Gamma_{0}&:=\{(u,v) \in \mathbb{R}^2 \, \,  
| \, \,  0 \leq u,\, \,  0 \leq v, \, \,  u+v<1\},\\
\Delta_{0}&:=\{(u,v) \in \mathbb{R}^2 \, \,  
| \, \,  0 < u,\, \,  0 < v, \, \, u+v \geq 1 , \, \, 2\omega(u,v) >n-1\},\\
\Gamma_{1}&:=\{(u,v) \in \mathbb{R}^2 \, \,  
| \, \,  0 \leq u,\, \,  v \leq 0, \, \,  u-v<1\},\\
\Delta_{1}&:=\{(u,v) \in \mathbb{R}^2 \, \,  
| \, \,  0 < u,\, \,  v < 0, \, \, 
u-v \geq 1 , \, \, 2\omega(u,v) >n-1\},
\end{split}
\end{equation}
and
\begin{equation}\label{eq:defomega1rec}
\omega(u,v):= \frac{n\, \arccos \left(\frac{u^2+v^2-1}{2u|v|}\right)+m\, \arccos \left( \frac{1-u^2+v^2}{2|v|}\right)}{2\pi}.
\end{equation}
Conversely, every trinomial of the form~\eqref{eq:schur} is Schur stable.
\end{theorem}
In a concious abuse of notation, after a natural identification, we use the same labels in~\eqref{eq:defgammadelta} and \eqref{eq:defgammadeltarec}.

We continue to rely on the notations and assumption of Theorem~\ref{th:main}.

\begin{remark}[The sppliting of the image]\label{rem:1}
\hfill

\noindent
We stress that $\Delta_{1}=\Delta_{n\mod 2}$ and $\Gamma_{1}=\Gamma_{n\mod 2}$ for an odd number $n$, and 
$\Delta_{0}=\Delta_{n\mod 2}$ and $\Gamma_{0}=\Gamma_{n\mod 2}$ for an even number $n$.
Moreover, Lemma~\ref{lem:piezas} in Appendix~\ref{ap:tools} implies
\begin{equation}
\Pi_n(\Omega_n)=
\begin{cases}
\Delta_0\cup \Gamma_0 & \textrm{ for $n$ being an even number},\\
\Delta_1\cup \Gamma_1 & \textrm{ for $n$ being an odd number}.
\end{cases}
\end{equation}
\end{remark}

\begin{remark}[Existence of a triangle]
\label{rem:2}
\hfill

\noindent
The Law of Cosines implies that for any $[1:x:y]\in \Delta_{n \mod 2}$ there exists a triangle with length sides $1$, $x=|x|$ and $|y|$.
\begin{figure}[!ht]
\includegraphics[scale=0.25]{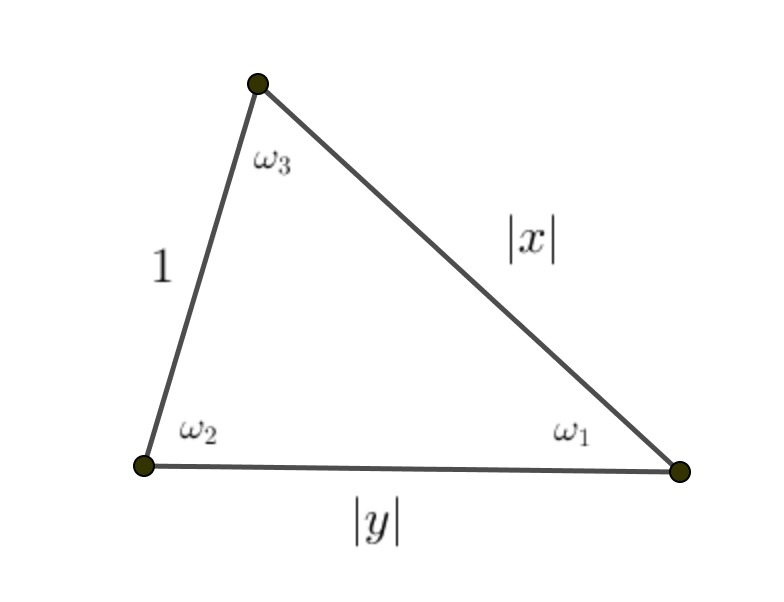}
\caption{Bohl's triangle for the stability region} \label{triangulo}
\end{figure}

\noindent
By the Law of Cosines  we have
\begin{equation}
\omega_1=\arccos\left(\frac{x^2+y^2-1}{2|x||y|}\right),\quad
\omega_2= \arccos\left(\frac{1+y^2-x^2}{2|y|}\right)
\end{equation}
and
\begin{equation}
\omega_3= \arccos\left(\frac{1+x^2-y^2}{2|x|}\right).
\end{equation}
Recall that $\omega_1+\omega_2+\omega_3=\pi$ and that $\arccos(-\theta)=\pi-\arccos(\theta)$ for $-1\leq \theta\leq 1$.
By~\eqref{eq:defomega1} we have that
\begin{equation}\label{ec:uno}
2\omega(x,y)= \frac{n\omega_1+m\omega_2}{\pi}>n-1
\end{equation} is equivalent to 
$n\omega_3+(n-m)\omega_2<\pi$.
We also note that the preceding inequality  is equivalent to
$n(\pi-\omega_1)-m\omega_2<\pi$.
\end{remark}

\begin{remark}[Schur stability for trinomials with real coefficients]
\hfill

\noindent
By Proposition~7.9 in~\cite{BarreraMaganaNavarrete} it follows that
$[1:x:y]\in \Pi_n(\Omega_n)\subset \Omega_n\cap \mathbb{P}^2_\mathbb{R}$ and hence it means that the associated trinomial $\zeta\mapsto\zeta^n+x\zeta^m+y$ is a Schur stable trinomial with $x,y\in \mathbb{R}$. 
We point out that the Schur stability for trinomials with general real coefficients has been established for instance 
in~\cite{CermakFedorkova2022}, Section~3 in~\cite{Cermak2022}, Theorem~2 in~\cite{Kipnis2004} and Theorem~1.3 
in~\cite{Cermak2015}.
\end{remark}

\begin{remark}[Dimension of $\Omega_n$ and the meaning of the parameters]\label{rem:dim}
\hfill

\noindent
We point out that the stability region  $\Omega_n$ is an open set in $\mathbb{C}^2$.
Hence, it can be naturally parametrized by four real parameters $(x,y,s,t)$.
To be more precise, 
\begin{itemize}
\item[(i)]
the parameters $x$ and $y$ are given by the condition $[1:x:y]\in \Pi_n(\Omega_n)$ in Theorem~\ref{th:main},
\item[(ii)]
the parameter $s$ is obtained 
by the angular flow defined in~\eqref{eq:angular} applied to $[1:x:y]\in \Pi_n(\Omega_n)$, that is,
\[
[0,2\pi]\ni s\mapsto
e^{\ii s} \cdot [1:x:y]:=[1: e^{-\ii(n-m)s}x: 
e^{-\ii ns}y],
\]
\item[(iii)] and finally the parameter $t$ satisfying (1) or (2) in Theorem~\ref{th:main} can be interpreted as the permissible variation of the pivot $P(t)$, starting with a trinomial in $\Pi _n(\Omega _n)$ which pivot $P(0)=P$ is given explicitly in~\eqref{eq:Pcasos}, in a way that the open interval given by Bohl's Theorem (Theorem~\ref{th:bohl}) contains exactly $n$ integers. By~\eqref{eP} and~\eqref{eq:Pcasos} we deduce that 
$|t|=\frac{2\pi}{n}|P(t)-P|$.
\end{itemize}
Moreover, it is an open contractible subspace of $\mathbb{R}^4$, in particular, it is path-connected, see Theorem~7.18 
in~\cite{BarreraMaganaNavarrete}.
\end{remark}

In Figure~\ref{pindeomegan} below, we plot $\Pi_n(\Omega_n)$ for the particular cases $(n,m)=(3,1)$; $(n,m)=(3,2)$; and $(n,m)=(4,3)$. We emphasize that the stability region for trinomials with real coefficients, that is $\Omega_n \cap \mathbb{P}^2_{\mathbb{R}}$, can be reconstructed from the projection $\Pi_n(\Omega_n)$ yielding the open set limited by the black curves and lines, see Figure~\ref{pindeomegan}. For instance, for $(n,m)=(4,3)$ we have the following:
\begin{itemize}
\item[(a)] For any $[1:x:y]\in \Delta_0\cup \Gamma_0$ the choice $t=0$ and $s=\pi$ yields
\[[1:-x:y]=[1:xe^{\ii t}e^{-\ii (n-m)s}:ye^{-\ii ns}]\in 
\Omega_n \cap \mathbb{P}^2_{\mathbb{R}}.
\]
\item[(b)] For any $[1:x:y]\in \Gamma_0$ the choice $t=\pi/4$ and $s=\pi/4$ implies
\[[1:x:-y]=[1:xe^{\ii t}e^{-\ii (n-m)s}:ye^{-\ii ns}]\in 
\Omega_n \cap \mathbb{P}^2_{\mathbb{R}}.
\]
\item[(c)] By Item~(b) we know that $[1:x:-y]\in \Omega_n \cap \mathbb{P}^2_{\mathbb{R}}$ whenever $[1:x:y]\in \Gamma_0$.
We note that the choice $t=0$ and $s=\pi$ gives
\[[1:-x:-y]=[1:xe^{\ii t}e^{-\ii (n-m)s}:-ye^{-\ii ns}]
\in \Omega_n \cap \mathbb{P}^2_{\mathbb{R}}
\]
when $[1:x:y]\in \Gamma_0$.
\item[(d)] 
If $[1:x:y]\in \Delta_0$ then $[1:x:-y]\not\in \Omega_n \cap \mathbb{P}^2_{\mathbb{R}}$. 
Recall that $n=4$ and $m=3$.
By contradiction, assume that there exist $|t|<\frac{\pi (2 \omega-n+1)}{n}$ and $0\leq s\leq 2\pi$ such that
\[
[1:x:-y]=[1:xe^{\ii t}e^{-\ii (n-m)s}:ye^{-\ii ns}].
\]
Then we have  $s\in \{\pi/4,(3/4)\pi, (5/4)\pi, (7/4)\pi\}$ and $t-(n-m)s=2\pi \ell$ for some integer $\ell$.
Since  $\omega\leq n/2$, we obtain $|t|=|s+2\pi\ell|<\frac{\pi (2 \omega-n+1)}{n}\leq \frac{\pi}{n}$, that is,
\[
|s+2\pi\ell|<\frac{\pi}{4}\quad  \textrm{ for }\quad  s\in \{\pi/4,(3/4)\pi, (5/4)\pi, (7/4)\pi\}\quad \textrm{ and }\quad \ell\in \mathbb{Z},
\]
which is a contradiction.
Similar reasoning implies that
$[1:-x:-y]\not\in \Omega_n \cap \mathbb{P}^2_{\mathbb{R}}$ whenever $[1:x:y]\in \Delta_0$.
\end{itemize}

\begin{figure}[!ht]
\includegraphics[width=0.90\textwidth]{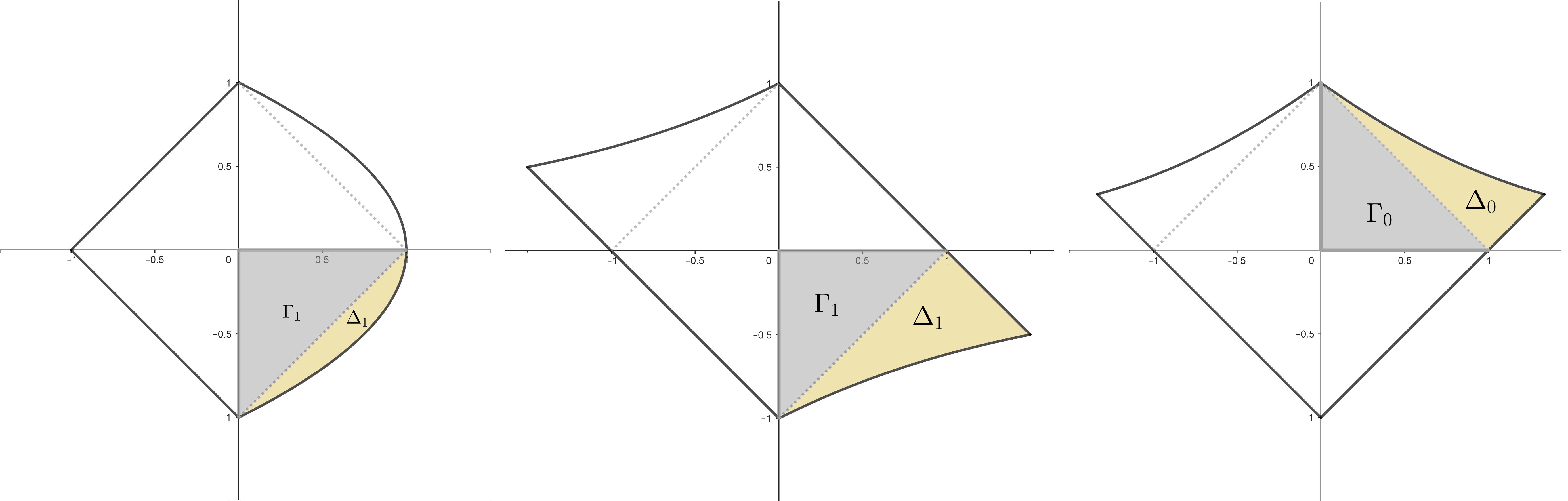}
\caption{The plot of $\Pi_n(\Omega_n)$ for the particular cases $(n,m)=(3,1), (3,2)$ and $(4,3)$, respectively.
The set $\Gamma_j$ is a right triangle that includes the legs sides, however, its hypotenuse (dotted line) belongs to $\Delta_j$. The solid curve bounding 
$\Delta_j$ does not belong to it.} \label{pindeomegan}
\end{figure}

\begin{remark}[Geometric interpretation of $\Gamma_0$, $\Gamma_1$, $\Delta_0$ and $\Delta_1$]\label{rem:geo}
\hfill

\noindent
It is well-known that $\Gamma_j$ and $\Delta_j$ belong to $\Omega_n \cap \mathbb{P}^2_{\mathbb{R}}$,
see Theorem~2 in~\cite{Kipnis2004} and the figure in p.~1712 
in~\cite{Kipnis2004}.
In fact, the (projective) Cohn domain $\mathcal{C}:=\{[1:u:v]\in \mathbb{P}^2_{\mathbb{R}}: |u|+|v|<1\}$ belongs to
$\Omega_n \cap \mathbb{P}^2_{\mathbb{R}}$. It corresponds to the region in $\mathbb{R}^2$, where the absolute value of the coefficients of the trinomial associated to the point $[1:x:y]$ are not the lengths of the sides of any triangle including degenerate triangles.
For $\Delta_j$ we obtain a geometric interpretation of it. To be precise, it is the region in $\mathbb{R}^2$, where the absolute value of the coefficients of the trinomial associated to the point $[1:x:y]$ are the lengths of the sides of some triangle (it may be degenerate). In addition, the dotted line in Figure~\ref{pindeomegan} represents when such triangle is degenerate.
\end{remark}

In the following corollaries  we always assume that  $\textsf{gcd}(n,m)=1$ and the notations of Theorem~\ref{th:main}.
\begin{corollary}[Not stable projection]\label{cor:proj}
\hfill

\noindent
If $[1:x:y]\not \in \Pi_n(\Omega_n)$ then  
$[1: xe^{\ii t} e^{-\ii(n-m) s} : ye^{-\ii ns}]\not \in \Omega_n$ for all $t\in \mathbb{R}$ and $s\in [0,2\pi]$.
In other words, if the projection of $[1:b:c]$ under $\Pi_n$ is not Schur stable, then $[1:b:c]$ is not Schur stable.
\end{corollary}

\begin{proof}
The proof is a direct consequence of Theorem~\ref{th:main}.
\end{proof}

\begin{remark}[Trinomials with complex coefficients which are not Schur stable]
\hfill

\noindent
Let $u$ and $v$ be positive numbers.
Assume that the trinomial $\zeta \mapsto \zeta^n+u\zeta^m+(-1)^n v$ is not Schur stable.
Then Corollary~\ref{cor:proj} yields that 
the trinomial $\zeta \mapsto \zeta^n+a\zeta^m+b$ is not Schur stable for
any complex numbers $a$ and $b$ satisfying $|a|=u$
 and $|b|=v$. If in addition, we assume that all roots of $\zeta \mapsto \zeta^n+u\zeta^m+(-1)^n v$ have modulus different from one, then
 the continuity of the roots of polynomials with respect to the coefficients yields that for any complex numbers $a$ and $b$ satisfying $|a|=u$
 and $|b|=v$ there exists $\varepsilon:=\varepsilon(a,b,n,m)>0$ such that the trinomials
$\zeta \mapsto \zeta^n+a_*\zeta^m+ b_*$ are not Schur stable for any complex numbers $a_*$ and $b_*$ satisfying $|a-a_*|<\varepsilon$ and $|b-b_*|<\varepsilon$. 
\end{remark}

\begin{remark}[Schur Stability of the projection $\Pi_n$ does not imply Schur stability]\hfill

\noindent
We stress that the converse statement of Corollary~\ref{cor:proj} is not true. In other words, in general  $\Pi_n[1:b:c] \in \Pi_n(\Omega_n)$ does not imply 
$[1:b:c] \in \Omega_n$.
For instance, let $n=4$ and $m=3$ and assume that $\zeta\mapsto\zeta^n+x\zeta^m+(-1)^n y$ with $x>0$ and $y>0$ is a Schur stable trinomial.
In addition, assume that $[1:x:(-1)^ny]\in \Delta_0$, see Figure~\ref{pindeomegan}.
Then the trinomials of the form
$\zeta \mapsto \zeta^n-x\zeta^m\pm y$ is not Schur stable.

Using a math software, one can verify that the trinomial $\zeta\mapsto \zeta^{11}-e^{\ii\cdot 0.6}\zeta^{10}-0.05e^{\ii\cdot 0.6}$ is not Schur stable even its corresponding projection
$\zeta\mapsto \zeta^{11}+\zeta^{10}-0.05$ is a Schur stable trinomial.
\end{remark}

\begin{remark}[Schur stability for real coefficients: characterization]\label{rem:cara}\hfill

\noindent
By Theorem~1.3 in~\cite{Cermak2015}
it follows that $[1:x:y] \in \Omega_n \cap \mathbb{P}^2_{\mathbb{R}}$ if and only if one of the following conditions is valid:
\begin{itemize}
\item[($C_1$)] $|x|+|y|<1$,
\item[($C_2$)] $|x|+|y|\geq 1$, $|x|-1<|y|<1$, $(-1)^{m} x^n y^{n-m}<0$ and
\begin{equation}\label{ec:cuatro}
\frac{n\, \arccos \left(\frac{1+x^2-y^2}{2|x|}\right)+(n-m)\, \arccos \left( \frac{1-x^2+y^2}{2|y|}\right)}{\pi}<1.
\end{equation}
\end{itemize}
Moreover, Corollary~2 in~\cite{Cermak2022} gives the following characterization of  $\Omega_n \cap \mathbb{P}^2_{\mathbb{R}}$. A point
$[1:x:y] \in \Omega_n \cap \mathbb{P}^2_{\mathbb{R}}$ if and only if one of the following conditions is valid: ($C_1$) given above or
\begin{itemize}
\item[($C^{\prime}_2$)] $|x|+|y|\geq 1$, $(-1)^{m} x^n y^{n-m}<0$ and
\begin{equation}\label{ec:cinco}
\frac{n\, \arccos \left(\frac{1-x^2-y^2}{2|x||y|}\right)-m\, \arccos \left( \frac{1-x^2+y^2}{2|y|}\right)}{\pi}<1.
\end{equation} 
\end{itemize}
We stress that~\eqref{ec:cuatro} 
and~\eqref{ec:cinco} are equivalent,
and the they are also equivalent to the condition $2\omega(x,y)>n-1$, where $\omega(x,y)$ is defined in~\eqref{eq:defomega1},
see Remark~\ref{rem:2}. 
\end{remark}

In the sequel, we show that Theorem~\ref{th:main} yields, in particular, a characterization of $\Omega_n \cap \mathbb{P}^2_{\mathbb{R}}$, i.e., the stability region for trinomials with general real coefficients.
We recall that
$\Pi_n([1:a:b])=[1:x: y]$, where $x=|a|>0$ and $y=(-1)^n|b|$, 
see~\eqref{def:pi},
and 
\begin{equation}
\Pi_n(\Omega_n)=
\begin{cases}
\Delta_0\cup \Gamma_0 & \textrm{ for $n$ being an even number},\\
\Delta_1\cup \Gamma_1 & \textrm{ for $n$ being an odd number},
\end{cases}
\end{equation}
see Remark~\ref{rem:1}.

\begin{corollary}[Schur stability for real coefficients]\label{cor:stablereal}
\hfill

\noindent
For any integer $n\geq 2$ it follows that 
$[1:x:y]\in \Gamma_{n\mod 2}$ if and only if $[1:x: -y],[1:-x:y], [1:-x:-y]\in \mathcal{C}:=\{[1:u:v]\in \mathbb{P}^2_{\mathbb{R}}: |u|+|v|<1\}$ with $x>0$ and $(-1)^n y>0$.
In addition,
\begin{itemize}
\item[(i)] for $n$ being an even positive integer it follows that 
 $[1:x:y]\in \Delta_0$  if and only if $[1:-x:y]\in (\Omega_n\cap\mathbb{P}^2_\mathbb{R})\setminus\mathcal{C}$ and $x>0$.
\item[(ii)] for $n\geq 3$ being an odd positive number and $m$ is an even positive number. Then  $[1:x:y]\in \Delta_1$  if and only if $[1:-x:-y]\in (\Omega_n\cap\mathbb{P}^2_\mathbb{R})\setminus\mathcal{C}$, $x>0$ and $y<0$.
\item[(iii)] for $n\geq 3$ being an odd positive number and $m$ is an odd positive number. Then $[1:x:y]\in \Delta_1$  if and only if $[1:x:-y]\in (\Omega_n\cap\mathbb{P}^2_\mathbb{R})\setminus \mathcal{C}$ and $y<0$.
\end{itemize}
\end{corollary}

\begin{proof}
The proof is given in Subsection~\ref{sub:proofcorstable} in Appendix~\ref{ap:tools}.
\end{proof}

As we have already pointed out in Remark~\ref{rem:geo} or in Corollary~\ref{cor:stablereal}, the Cohn domain $\mathcal{C}\subset \Omega_n \cap \mathbb{P}^2_{\mathbb{R}}$. However, its boundary
$\partial \mathcal{C}:=\{[1:u:v]\in \mathbb{P}^2_{\mathbb{R}}: |u|+|v|=1\}$ does not belong in $\Omega_n \cap \mathbb{P}^2_{\mathbb{R}}$. By Remark~\ref{rem:cara} or Corollary~\ref{cor:stablereal} we have that 
$\{[1:u:v]\in \mathbb{P}^2_{\mathbb{R}}: |u|+|v|=1,\,(-1)^{m} u^n v^{n-m}<0\}$ belongs to $\Omega_n \cap \mathbb{P}^2_{\mathbb{R}}$.

The following corollary yields that the trinomial $\zeta\mapsto\zeta^n+b\zeta^m+c$ with $1\leq m< n-1$, and $b,c\in \mathbb{C}$ satisfying $|b|=1$ and $0<|c|<1$ is never Schur stable.

\begin{corollary}[Trinomials with two unimodular coefficients]\label{cor:uno}
\hfill

\noindent
Assume that $|b|=1$ and 
$0<|c| < 1$. If $1\leq m<n-1$, then 
$\Pi _n([1:b:c])=[1: 1:(-1)^n|c|]\not\in \Omega_n$. Hence $[1:b:c] \notin \Omega _n$. 
\end{corollary}
\begin{proof}
Assume that $0<|c|<1$.
We note that 
\begin{equation}\label{eq:crest}
\frac{\pi}{3}\leq \arccos\left(\frac{|c|}{2}\right)<\frac{\pi}{2}.
\end{equation}
We start with the case when $n$ is an even number.
The case when $n$ is an odd number is analogous.
By~\eqref{def:pi} we have that $\Pi_n([1:b:c])=[1:1,|c|]$. Now, we verify when $[1:1:|c|]\in \Delta_0$, where $\Delta_0$ is defined in~\eqref{eq:defgammadelta}.
For short, we write $\omega:=\omega(1,|c|)$.
Observe that $2\omega>n-1$ if and only if
\begin{equation}
2\omega=(n+m)\frac{\arccos\left(\frac{|c|}{2}\right)}{\pi}>n-1.
\end{equation}
The preceding inequality together with~\eqref{eq:crest} imply
\[
(n+m)\frac{1}{2}>n-1
\]
yielding $m=n-1$. By Corollary~\ref{cor:proj} we conclude the second part of the statement. 
\end{proof}

Corollary~\ref{cor:uno} implies that the trinomial associated to 
$[1:\pm 1:c]$ is stable only when $m=n-1$.
In Theorem~2 in~\cite{Kuruklis}, the author studies 
the Schur stability for the trinomial associated to $[1:-1:c]$ with $c\in \mathbb{R}$ and $m=n-1$.
The following corollary yields the Schur stability for the trinomial associated to $[1:1:c]$.

\begin{corollary}[Theorem~2 in~\cite{Kuruklis}: real case]\label{cor:dos}
\hfill

\noindent
Assume that $b=1$, $0<c<1$ and $m=n-1$.
\begin{enumerate}
\item[(i)] If $n$ is an odd number then $[1:1:c]\not\in \Omega_n$. Moreover, $[1:1:-c]\in  \Omega_n$ if and only if 
\begin{equation}\label{eq:upper}
\frac{(n-1)\pi}{2n-1} < \arccos \left(\frac{c}{2}\right) < \frac{\pi}{2}.
\end{equation}
\item[(ii)] If $n$ is an even number then
$[1: 1 :-c]\not\in \Omega_n$. In addition, 
$[1: 1 :c]\in \Omega_n$  if and only if
$$\frac{(n-1)\pi}{2n-1} < \arccos \left(\frac{c}{2}\right) < \frac{\pi}{2}.$$ 
\end{enumerate} 
\end{corollary}

As a consequence of Corollary~\ref{cor:uno} we have that  $n=2$ and $c\in (-1,1)$ are neccesary conditions in order that the Lambert trinomial (see~\cite{Wang}) $L_{n}(\zeta)=\zeta^n+\zeta+c$, $\zeta\in \mathbb{C}$ with $c\in \mathbb{R}$ is Schur stable. However, such conditions are not sufficient. Indeed, by Corollary~\ref{cor:dos} we have that  $L_2$ is a Schur stable trinomial if and only if  $0<c<1$.

\begin{proof}[Proof of Corollary~\ref{cor:dos}]
We start with the proof of Item (i).
By~\eqref{eq:Pcasos} we note that $P=1/2$. Then
Theorem~\ref{th:bohl} with the help of Lemma~\ref{parity} in Appendix~\ref{ap:tools} implies the first part of the statement.

We continue with the second part of the statement. By Theorem~\ref{th:main} it is enough to show that $\Pi_n([1:1:-c])=[1:1:-c]\in \Delta_1$.
For short, we write $\omega:=\omega(1,|c|)$.
Observe that $2\omega>n-1$ if and only if
\begin{equation}
2\omega=(2n-1)\frac{\arccos\left(\frac{|c|}{2}\right)}{\pi}>n-1
\end{equation}
yields the left-hand side of~\eqref{eq:upper}.
Since $c\in (0,1)$, the right-hand side of~\eqref{eq:upper} follows straightforwardly.

The proof of Item (ii) is similar and we omit it.
\end{proof}

The following corollary yields the Schur stability for the trinomial associated to $[1:b:c]$ for  $m=n-1$ and complex coefficients $|b|=1$ and $0<|c|<1$. In particular, it implies Theorem~2 in~\cite{Kuruklis}, and
Theorem~1 and Theorem~2 in~\cite{Matsunaga2010}.

\begin{corollary}[Theorem~2 in~\cite{Kuruklis}, Theorem~1 and Theorem~2 in~\cite{Matsunaga2010}:
 complex case]\label{cor:complex}
\hfill

\noindent
Assume that $|b|=1$, $0<|c|<1$ and $m=n-1$.
The point $[1:b:c] \in \Omega _n$ if and only if 
$$\left(\frac{n-1}{2n-1} \right) \pi<\arccos\left(\frac{|c|}{2}\right)< \frac{\pi}{2},$$
and $[1:b:c]$ can be represented in the form  
\[
[1:e^{\ii t}e^{-\ii s}:(-1)^n|c|e^{-\ii ns}], \quad \textrm{ where }\quad s\in[0,2\pi]
\]
and 
\[
|t|<\frac{(2n-1)\arccos\left(\frac{|c|}{2}\right)-(n-1)\pi}{n}.
\]
\end{corollary}

\begin{proof}
The proof is a direct consequence of Theorem~\ref{th:main}.
\end{proof}

\begin{remark}[A discussion about trinomials with complex coefficients]
\hfill

\begin{itemize}
\item[(i)]
We stress that Theorem~\ref{th:main} generalizes the results given 
in~\cite{Cermak2022} for trinomials with real coefficients to the setting of complex coefficients.
\item[(ii)]
We point out that, in particular, Theorem~\ref{th:main} gives the stability region for trinomials with real coefficients, see Remark~\ref{rem:cara} and Corollary~\ref{cor:stablereal}.
\item[(iii)] 
Using the so-called Schur--Cohn method, necessary and sufficient analytic conditions for the Schur stability of 
trinomials with complex coefficients and exponents $n$ general and $m=n-1$
 have been analyzed in Theorem~2 and Theorem~4 of~\cite{Cermak2014new}, see also Theorem~14 in~\cite{Cermak2015survey}.
More recently, using the so-called discrete D-decomposition method, necessary and sufficient analytic conditions for the Schur stability of trinomials with complex coefficients and 
general exponents $n$ and $m$ 
is provided in Theorem~1 of~\cite{Cermak2018JiriMat}.
We point out that our main result Theorem~\ref{th:main} is an implicit geometric parametrization of the stability region, see Remark~\ref{rem:dim}.
 
In what follows, we compare the results of Theorem~1 in \cite{Cermak2018JiriMat} with the findings of Theorem~\ref{th:main}.
We verify that Theorem~1 in \cite{Cermak2018JiriMat} implies Theorem~\ref{th:main}. Inspecting the proof below, one can see that the converse also holds true.
It is not hard to see that the condition~(4) in 
Theorem~1 of \cite{Cermak2018JiriMat}
is equivalent  to the Item~(1) in Theorem~\ref{th:main}.
In the sequel, we assume that condition given in~(5)-(6) of Theorem~1 in~\cite{Cermak2018JiriMat} is valid. They read for the trinomial~\eqref{def:fa} as follows:
$|b|+|c|\geq 1$, $|b|-1<|c|<1$ and
\begin{equation}\label{eq:cer}
\begin{split}
\hspace{2.0cm}n\,\arccos\left(\frac{1+|b|^2-|c|^2}{2|b|}\right)&+(n-m)\,\arccos\left(\frac{1-|b|^2+|c|^2}{2|c|}\right)\\
&<\arccos\left(\cos\left(n\arg(-b)-(n-m)\arg(-c)\right)\right),
\end{split}
\end{equation}
where $\arg(z)$ denotes the argument of a given non-zero complex number $z$.
Let $x=|b|$ and $y=|c|$ and recall the definition of the angles $\omega_1$, $\omega_2$ and $\omega_3$ given in Figure~\ref{triangulo} of Remark~\ref{rem:2}.
The inequality~\eqref{eq:cer} reads as follows
\[
n\omega_3+(n-m)\omega_2<\arccos\left(\cos\left(n\arg(b)-(n-m)\arg(c)+m\pi\right)\right),
\]
where we have used the fact that  $\arg(-z)=\pi+\arg(z)$, $z\in \mathbb{C}$, $z\neq 0$.
Since $\omega_1+\omega_2+\omega_3=\pi$
and $2\pi\omega=n\omega_1+m\omega_2$ (recalling $\omega:=\omega(1)$ for $r=1$ defined in~\eqref{def:PW}),
 we obtain
\begin{equation}\label{ec:one}
n\pi-2\pi\omega<\arccos\left(\cos\left(n\arg(b)-(n-m)\arg(c)+m\pi\right)\right).
\end{equation}
Since the Cosine function is even, we also have
\begin{equation}\label{ec:dos}
n\pi-2\pi\omega<\arccos\left(\cos\left(-n\arg(b)+(n-m)\arg(c)-m\pi\right)\right).
\end{equation}
Now, assume that $n$ is an even number. Since $\gcd(n,m)=1$, we have that $m$ is an odd number.
Recall that $\arccos(-\mu)=\pi-\arccos(\mu)$ for $-1\leq \mu\leq 1$ and $\cos(\varphi+k\pi)=-\cos(\varphi)$ for $\varphi\in \mathbb{R}$ and $k$ being an odd integer number.
Then~\eqref{ec:one} is equivalent to
\begin{equation}\label{ec:onenew}
\begin{split}
n\pi-2\pi\omega&<\arccos\left(-\cos\left(n\arg(b)-(n-m)\arg(c)\right)\right)\\
&=\pi-\arccos\left(\cos\left(n\arg(b)-(n-m)\arg(c)\right)\right)\\
&=\pi-n\arg(b)+(n-m)\arg(c),
\end{split}
\end{equation}
and~\eqref{ec:dos} is equivalent to
\begin{equation}\label{eq:dosnew}
\begin{split}
n\pi-2\pi\omega&<\arccos\left(-\cos\left(-n\arg(b)+(n-m)\arg(c)\right)\right)\\
&=\pi-\arccos\left(\cos\left(-n\arg(b)+(n-m)\arg(c)\right)\right)\\
&=\pi+n\arg(b)-(n-m)\arg(c).
\end{split}
\end{equation}
Note that~\eqref{ec:onenew} is equivalent to
\begin{equation}
\frac{n\arg(b)-(n-m)\arg(c)}{2\pi}<\frac{1}{2}(2\omega-(n-1))
\end{equation}
and~\eqref{eq:dosnew} is equivalent to
\begin{equation}
-\frac{1}{2}(2 \omega-(n-1))<\frac{n\arg(b)-(n-m)\arg(c)}{2\pi}.
\end{equation}
The preceding two inequalities are equivalent to
\begin{equation}\label{eq:nodo}
\frac{1}{2}(2 \omega-(n-1))>\left|\frac{n\arg(b)-(n-m)\arg(c)}{2\pi}\right|.
\end{equation}
By~\eqref{eP} and~\eqref{eq:Pcasos}
the corresponding pivots for the points $[1:a:b]$ and $[1:x:y]$ are given by
\begin{equation}
P=\frac{n(\pi+\arg(b)-\arg(c))-m(\pi-\arg(c))}{2\pi}\quad \textrm{ and }\quad
P_*:=\frac{n-m}{2},
\end{equation}
respectively.
Note that
\begin{equation}
\begin{split}
P-P_*&=\frac{n(\pi+\arg(b)-\arg(c))-m(\pi-\arg(c))}{2\pi}-\frac{(n-m)\pi}{2\pi}\\
&=\frac{n(\arg(b)-\arg(c))+m\arg(c)}{2\pi}\\
&=\frac{n\arg(b)-(n-m)\arg(c)}{2\pi}.
\end{split}
\end{equation}
Therefore,~\eqref{eq:nodo} reads as follows
\[
|P-P_*|<\frac{1}{2}(2\omega-(n-1)).
\]
By Item~(iii) in~Remark~\ref{rem:dim} we obtain
\[
|t|< \frac{2\pi}{n}\cdot\frac{1}{2}(2\omega-(n-1))=\frac{\pi}{n}(2\omega-(n-1)),
\] 
which is the statement of Item~(2) in Theorem~\ref{th:main}.

The case when $n$ is an odd number is analogous.
\end{itemize}
\end{remark}

The rest of the manuscript is organized as follows. 
In Section~\ref{sec:proof} we provide the proof of  Theorem~\ref{th:main}.
Finally, 
in Appendix~\ref{ap:tools} 
we state and show auxiliary results that we use throughout the manuscript. In addition, in Subsection~\ref{sub:proofcorstable} of Appendix~\ref{ap:tools} we give the proof of Corollary~\ref{cor:stablereal}.
In Appendix~\ref{ap:Bohlstatement}, for completeness of the presentation, we show that the exceptional cases in the original statement of Bohl's Theorem given in~\cite{Bohl} are included in the counting procedure of Theorem~\ref{th:bohl}.

\section{\textbf{Proof of the main result: Theorem~\ref{th:main}}}\label{sec:proof}
The proof of Theorem~\ref{th:main} is given in Subsection~\ref{sub:proof}. It relies on
Theorem~\ref{th:bohl} and  basic arithmetic properties about open intervals (the number of integers that they contain). This is the content of Lemma~\ref{basicfactsonintervals1}, Lemma~\ref{basicfactsonintervals2} and Lemma~\ref{basicfactsonintervals3}, which is proved in Subsection~\ref{sub:ari}.

\subsection{\textbf{Arithmetic properties for open intervals}}\label{sub:ari}
From now on, unless otherwise  specified, 
\[
I=(P-\omega, P+\omega)
\] 
denotes an open interval with center (pivot)  $P\in \mathbb{R}$ and radius $\omega \geq 0$ with the understanding $I=\emptyset$ for $\omega=0$ and any $P\in \mathbb{R}$. The numbers $P-\omega$ and $P+\omega$ are called the boundary points of $I$.
For shorthand, we denote the length of $I$ by $\mu(I)$  and note that $\mu(I)=2\omega$.
Moreover, we denote by $\#I$ the cardinality of the set $I\cap \mathbb{Z}$.

In the sequel, we show that the length of an open interval $I$ with no integer boundary points and  containing precisely $k$ integers satisfies $k-1<\mu(I)<k+1$. We point that  that for $k=0$, the lower bound $k-1=-1$ is not informative.
This is the content of the following lemma.
\begin{lemma}[Localization of the lenght given the cardinality]\label{basicfactsonintervals1} 
\hfill

\noindent
Assume that $\#I=k$ for some $k\in \mathbb{N}\cup \{0\}$ and the boundary points of $I$ are not integers, then it follows that $k-1 < 2\omega < k+1$.
\end{lemma}
\begin{proof}
By hypothesis there exists a unique $\ell\in \mathbb{Z}$ such that 
\begin{equation}\label{eq:uno}
\ell-1<P-\omega< \ell
\quad \textrm{ and }\quad
\ell+k-1<P+\omega<\ell+k.
\end{equation}
Since $2\omega=(P+\omega)-(P-\omega)$,~\eqref{eq:uno} implies
\begin{equation}
\begin{split}
k-1=\left(\ell +(k-1)\right)-\ell <2\omega<\ell +k -(\ell -1)=k+1.
\end{split}
\end{equation}
The preceding inequality concludes the statement.
\end{proof}

\begin{remark}[Boundary points of $I$]\label{freebdrypoints}
\hfill

\noindent 
If in Lemma~\ref{basicfactsonintervals1}
we omit the restriction that the boundary points of $I$ are not integers, then it follows that
$k-1 < 2\omega \leq k+1$. 
\end{remark}
Now, we study in detail the following question.
Given an open interval $I$ satisfying $k-1<\mu(I)<k+1$, 
\textit{how many integers does it contain?}
It is not hard to see that there are three possible values: $k-1$, $k$ or $k+1$. Moreover, once the length of the interval  $I$ is fixed, there are only two possible values depending on
whether $k-1<\mu(I)\leq k$ or $k<\mu(I)<k+1$.
In the first case, the possible values are $k-1$ and $k$, while in the second case, the possible values are $k$ and $k+1$. This is the content of the following lemma.

\begin{lemma}[Localization of the cardinality given the length]\label{basicfactsonintervals2} 
\hfill

\noindent
Let $k$
be a non-negative integer such that
$k-1<2\omega <k+1$ ($0\leq 2\omega < 1$ for $k=0$) then $\#I \in \{k-1, k,k+1\}$ ($\#I \in \{0, 1\}$ for $k=0$). 
In addition,
\begin{enumerate}
\item[(i)] if $k-1<2\omega <k$ (equivalently, if $2\omega -(k-1) < k+1-2\omega$) 
then $\#I\in \{k-1,k\}$.
\item[(ii)] if $2\omega = k$ (equivalently, if $2\omega -(k-1) = k+1-2\omega$) then $\#I=k$ or the boundary points of $I$ are integers and 
$\#I=k-1$.
\item[(iii)] if  $k< 2\omega <k+1$ (equivalently, if $2\omega -(k-1) > k+1-2\omega$)  then 
$\#I\in \{k,k+1\}$.
\end{enumerate}
\end{lemma}
\begin{proof} 
The proof is done by contradiction. 
Let us assume that $\#I=k+j$ for some $|j|>1$. 
By Remark~\ref{freebdrypoints} we have that
\[k+j-1<2\omega \leq k+j+1,\]
which leads to a contradiction to the hypothesis $k-1<2\omega <k+1$.
As a consequence, $\#I \in \{k-1, k,k+1\}$.

In the sequel, we show Item~(i). Assume that $k-1< 2\omega <k$. 
We claim that 
no interval of length $2\omega$ contains exactly $k+1$ integers. Indeed, by Remark~\ref{freebdrypoints} we have that  $k<2\omega \leq k+2$ and it contradicts the assumption that $k-1< 2\omega <k$.

Now, we provide two examples of open intervals $I_1$ and $I_2$ of length $2\omega$ with precisely $k$ and $k-1$ integers, respectively.
For instance, the open intervals 
\[
I_1:=\left( \frac{k-1}{2}-\omega , \, \, \omega + \frac{k-1}{2}\right), \quad I_2:=\left( 0, 2\omega\right)
\quad 
\]
have length $2\omega$, $I_1\cap\mathbb{Z}=\{0, \ldots, k-1\}$ and $I_2\cap\mathbb{Z}=\{1, \ldots, k-1\}$.

The proofs of Item~(ii) and Item~(iii) are analogous and we omit them.
\end{proof}

\begin{remark}[At most three different values given the length]
\hfill

\noindent
Roughly speaking, Lemma~\ref{basicfactsonintervals2} can be interpreted as follows. By moving the center (pivot) $P$ of the original interval $I$ and preserving its length $2\omega$, it follows that
the amount of integers in the new interval diminishes by one, does not change or increases by one. In addition, Lemma~\ref{basicfactsonintervals2} yields precisely when the new interval remains with the same amount of integers, decreases by one or increases by one, according between which integers the length of the interval is located. 
\end{remark}

Broadly speaking, the next lemma allows us to quantify how far we can move the initial pivot $P \in \mathbb{Z}/2:=\{\ell /2 : \ell \in \mathbb{Z}\}$ of the open interval $I$ with no integer boundary points, preserving the length of the interval and the number of integers contained in the new interval, which  also has no integer boundary points.

\begin{lemma}[Localization of the cardinality given the lenght and the pivot]\label{basicfactsonintervals3} 
\hfill

\noindent
Assume that $\#I=k\geq 0$, the boundary points of $I$ are not integers and 
$P \in \mathbb{Z}/2$.  We consider the open interval
$I':=(P'-\omega, P'+\omega)$ and define
\[
\nu_1:=2\omega -(k-1),\qquad \nu_2:=
k+1-2\omega.
\]
Then it follows that
\begin{enumerate}
\item[(i)] If  $0\leq k-1<2\omega <k$ and $|P-P'|<\nu_1/2$ then $\#I'=k$ and no boundary point of $I'$ is an integer.
\item[(ii)] If  $0\leq k-1<2\omega <k$ and $|P-P'|=\nu_1/2$ then some boundary point of $I'$ is integer and $\#I'=k-1$.
\item[(iii)] If  $k=2 \omega$ and $|P-P'|<\nu_1/2=\nu_2/2$ then $\#I'=k$ and no boundary point of $I'$ is an integer.
\item[(iv)] If  $k=2 \omega$ and $|P-P'|=\nu_1/2=\nu_2/2$ then both boundary points of $I'$ are integers and $\#I'=k-1$.
\item[(v)] If $k<2\omega < k+1$ and  $|P-P'|<\nu_2/2$ then $\#I'=k$ and no boundary point of $I'$ is an integer.
\item[(vi)] If  $k<2\omega < k+1$ and  $|P-P'|=\nu_2/2$ then some boundary point of $I'$ is an integer and $\#I'=k$.
\end{enumerate}
\end{lemma}
\begin{proof}
We start with the proof of Item~(i).
The proof of Item~(i) is divided in two cases:  $P\in \mathbb{Z}$ or $P\in (\mathbb{Z}+1/2)$.

First, we assume that $P \in \mathbb{Z}$.
By Item~(i) of Lemma~\ref{parity} in Appendix~\ref{ap:tools}, we have that $k=2k'+1$ for some non-negative integer $k'$.
 Moreover, $$I\cap \mathbb{Z}=
 \{P-k', P-k'+1, \ldots, P,\ldots, P+k'-1, P+k'\}.$$
Since $|P'-P|<\nu_1/2=\omega-\left(\frac{k-1}{2}\right)=\omega-k'$, we have that
\begin{equation}\label{uno}
P'-\omega<P-k' \quad \textrm{ and } \quad 
 P+k'< P'+\omega,
\end{equation}
Then
$|P'-P|<k'+1-\omega$ due to $2\omega <k=2k'+1$.
The latter yields
\begin{equation}\label{cuatro}
 P-(k'+1)  <  P'-\omega \quad \textrm{ and }\quad
 P'+\omega <  P+k'+1.
\end{equation}
By~\eqref{uno} to~\eqref{cuatro} we deduce that \[I' \cap \mathbb{Z}=\{P-k', P-k'+1,\ldots,P,\ldots, P+k'-1, P+k'\}\] and $P'-\omega$, $P'+\omega$ are not integers. Therefore, 
$\#I'=2k'+1=k$. 

We continue the proof with the case  $P=\ell +1/2$ for some $\ell \in \mathbb{Z}$. 
By Item~(ii) of Lemma~\ref{parity} in Appendix~\ref{ap:tools} we have that 
$\#I=2k'$ for some non-negative integer $k'$.
Moreover, 
\[
I\cap \mathbb{Z}=\{\ell-k'+1, \ell -k'+2, \ldots , \ell+k'-1,\ell+k'\}.
\]
Since $|P'-(\ell+1/2)|=|P'-P|<\omega-k'+1/2$, it follows that
\begin{equation}\label{cinco}
P'-\omega  <  \ell-k'+1\quad \textrm{ and }\quad
 \ell+k' < P'+\omega .
\end{equation}
The hypothesis $2\omega <k=2k'$ also reads as $\omega -k'+1/2<k'+1/2-\omega$. As consequence,
$|P'-(\ell+1/2)|<\omega-k'+1/2<k'+1/2-\omega$, which implies
\begin{equation}\label{ocho}
 \ell-k'  <  P'-\omega\quad 
 \textrm{ and }\quad
 P'+\omega  <  \ell+k'+1.
\end{equation}
By~\eqref{cinco} and~\eqref{ocho} it follows that 
\[
I' \cap \mathbb{Z}=
\{\ell-k'+1, \ell -k'+2, \ldots , \ell+k'\}
=I \cap \mathbb{Z},
\] 
and the numbers $P'-\omega$ and $P'+\omega$ are not integers.
This completes the proof of Item~(i).

We continue with the proof of Item~(ii). As in the proof of Item~(i), the proof of Item~(ii) is divided in two cases:  $P\in \mathbb{Z}$ or $P\in (\mathbb{Z}+1/2)$.
We point out that $P'\neq P$ due to the assumptions.

We start assuming $P \in \mathbb{Z}$.
By Item~(i) of Lemma~\ref{parity} in Appendix~\ref{ap:tools}, we have that $k=2k'+1$ for some non-negative integer $k'$. In addition, 
\[
I\cap \mathbb{Z}=\{P-k', P-k'+1, \ldots,P,\ldots,P+k'-1, P+k'\}.
\]
The hypothesis $ |P'-P|=\nu_1/2=\omega-k'>0$ yields
$P'-P=\omega-k'$ or $P'-P=-(\omega-k')$. If $P'-P=\omega -k'$ then 
\begin{equation}\label{nueve}
P'-\omega =  P-k' \in \mathbb{Z}.
\end{equation}
The inequality $P-P'<0<\omega-k'$ implies that 
\begin{equation}\label{eq:media}
P+k'<P'+\omega.
\end{equation}
By hypothesis $2\omega <k=2k'+1$, which implies $\omega -k'<k'+1-\omega$. As a consequence,
$|P-P'|=\omega-k'<k'+1-\omega$, which yields 
\begin{equation}\label{eq:ultima}
P'+\omega < P+k'+1.
\end{equation}
By~\eqref{nueve},~\eqref{eq:media} 
and~\eqref{eq:ultima} we deduce that 
\[
I'\cap \mathbb{Z}:=\{P-k'+1,P-k',\ldots,P,\ldots,
P+k'-1,P+k'\} 
\] and $P'-\omega$ is an integer. The proof for the case $P'-P=-(\omega-k')$ is analogous and we omit it.

We continue the proof with the case  $P=\ell +1/2$ for some $\ell \in \mathbb{Z}$.
By Item~(ii) of Lemma~\ref{parity} in Appendix~\ref{ap:tools} we have that $k=2k'$ for some non-negative integer $k'$.
Moreover, 
\[
I\cap \mathbb{Z}=\{\ell-k'+1, \ell -k'+2, \ldots , \ell+k'-1,\ell+k'\}.
\]
As in the previous case, we have that $P'-(\ell+1/2)=P'-P=\omega-k'+1/2>0$ 
or $P'-P=-(\omega-k'+1/2)<0$.
We show the case $P'-P=\omega-k'+1/2$. The proof for the case $P'-P=-(\omega-k'+1/2)$ is analogous and we omit it.

We note that 
$P'-(\ell+1/2)=P'-P=\omega-k'+1/2>0$
 implies 
\begin{equation}\label{eq:ten}
P'-\omega=\ell-k'+1.
\end{equation} 
In addition, the inequality $\ell+1/2-P'=P-P'<0<\omega-k'+1/2$ implies that
 \begin{equation}\label{trece}
 \ell+k' < P'+\omega .
\end{equation}
By hypothesis $2\omega <k=2k'$, we have that $\omega -k'+1/2<k'+1/2-\omega$. Thus,
\[
|P'-\ell-1/2|=|P'-P|=\omega-k'+1/2<k'+1/2-\omega,
\]
which implies
\begin{equation}\label{catorce}
 P'+\omega <  \ell+k'+1.
\end{equation}
By~\eqref{eq:ten},~\eqref{trece},~\eqref{catorce} we obtain 
\[
I'\cap \mathbb{Z}=\{\ell -k'+2, \ldots , \ell+k'-1,\ell+k'\}
\]
and $P'-\omega$ is an integer. 
The proof of Item~(ii) is complete.

The proofs of Item~(iii), Item~(iv), Item~(v) and Item~(vi) are similar and we left the details to the interested reader. 
\end{proof}

In the following lemma, Item~(1) and 
Item~(2) are restatements of 
Lemma~\ref{basicfactsonintervals3} in a condensed form. We state them here for completeness of the presentation.
 Additionally, we introduce 
 Item~(3) which allows us to see when the new interval $I'$ has different cardinality from the initial interval $I$. 
\begin{lemma}[Variation of the pivot]\label{formareducida}
\hfill

\noindent
Assume that $\# I =k$, the boundary points of $I$
are not integers and $P \in \mathbb{Z}/2:=\{ \ell/2 : \ell \in 
\mathbb{Z}\}$.  We denote the open interval $(P'-\omega, P'+\omega)$ 
by $I'$, and 
\begin{equation}\label{eq:nu}
\nu = \min \left\{ 2\omega -(k-1), 
k+1-2\omega \right\}.
\end{equation}
\begin{enumerate}
\item If  $|P-P'| <\frac{\nu}{2}$ then $\#I' =k$ and no boundary point of $I'$ is an integer.
\item If  $|P-P'|=\frac{\nu}{2}$ then a boundary point of $I'$ is integer. 
\item If $\frac{\nu}{2} < |P-P'| \leq \frac{1}{2}$ then $\# I' \ne k$
and no boundary point of $I'$ is an integer.
\end{enumerate}
\end{lemma}
\begin{proof}
We prove Item~(3). For $P_0:=P-1/2$ and $Q_0 := P+1/2$, Lemma~\ref{parity} in Appendix~\ref{ap:tools} yields that the intervals $I_0=(P_0 -\omega, P_0+\omega)$ and
$J_0=(Q_0-\omega, Q_0+\omega)$ satisfy
$\# I_0 = \# J_0 \ne k$. 

By hypothesis $\frac{\nu}{2}<|P-P'|\leq \frac{1}{2}$ (so $\nu<1$, 
and it means that no boundary point of $I_0$ or $J_0$ is an integer), which implies
\[|P_0-P'|< \frac{1-\nu}{2}  \quad\textrm{ or } \quad |Q_0-P'|< \frac{1-\nu}{2}.\] 
In the sequel, we show that $\#I'=\# I_0=\# J_0$. 
Since $\#I_0\neq k$ and $\mu(I_0)=2\omega$, 
Lemma~\ref{basicfactsonintervals2} implies that $\#I_0\in \{k-1,k+1\}$.
Assume that $\#I_0=k-1$. Since $\mu(I_0)=2\omega$,
Lemma~\ref{basicfactsonintervals1} yields $k-2<2\omega <k$.
By hypothesis $\#I=k$ and 
$\mu(I)=2\omega$, then 
Lemma~\ref{basicfactsonintervals1} implies $k-1<2\omega <k+1$. Therefore, $k-1<2\omega<k$
and hence $\nu_0=k-2\omega$,
where $\nu_0$ is the corresponding value 
of~\eqref{eq:nu} for $I_0$.
Since
\[
1-\nu=1-(2\omega -(k-1))=\nu_0,
\]
Item~(1) of Lemma~\ref{formareducida} applied to $I_0$ implies $\# I'= \#I_0=k-1$ and no boundary point of $I'$ is an integer.
The case $\#I_0=k+1$ is analogous.
In summary, $\# I '\ne k$ and no boundary point of $I'$ is an integer.
\end{proof}

\subsection{\textbf{Proof of Theorem~\ref{th:main}}}\label{sub:proof}

In this subsection, we show Theorem~\ref{th:main}, which is a
consequence of what we have already proved up to here.

\begin{proof}[Proof of Theorem~\ref{th:main}]
Without loss of generality, we assume that $n$ is even number (the proof for the case $n$ is an odd number is analogous). 

We start showing the necessary implication.
Let $[1:b:c] \in \Omega _n$ be fixed and write $b=xe^{\ii \theta_1}$ and $c=ye^{\ii \theta_2}$ for some $\theta_1,\theta_2\in \mathbb{R}$, where $x=|b|$, $y=|c|$.
By Proposition~7.9 in~\cite{BarreraMaganaNavarrete} we have that
\begin{equation}\label{eq:igualdad}
[1:x:y]:=[1:|b|:|c|] \in \Pi_n(\Omega _n)\subset \Omega_n\cap \mathbb{P}^2 _{\mathbb{R}},
\end{equation} 
where $\Pi_n$ is defined in~\eqref{def:pi}.
For convenience, we introduce the  following notation
\begin{equation}\label{eq:action}
e^{\ii \theta} \cdot [1: b: c] := [1:b e^{-\ii(n-m)\theta}: c e^{-in \theta}].
\end{equation}
In particular, for $\sigma = \frac{\theta _2}{n}$ we have that
\begin{equation}\label{eq:sigmaact}
e^{\ii \sigma}\cdot [1:b:c]=[1:x e^{\ii t^\prime}:y],
\end{equation}
where $t^\prime:=\theta_1-(n-m) \frac{\theta _2}{n}$.
Since $n$ and $m$ are coprime numbers, we have that
$n$ and $n-m$ are also coprime numbers. Then there exists (at most two) an $n$-th root of unity 
$e^{2\pi \ii k /n}$ such that
\begin{equation}\label{ec:dio}
\left(e^{\frac{2\pi \ii k}{n}}\right) ^{-(n-m)} e^{\ii t^\prime}= e^{\ii t},\quad 
\textrm{ where }\quad |t|\leq \frac{\pi}{n}.
\end{equation}
By~\eqref{eq:sigmaact} and~\eqref{ec:dio} we obtain
\begin{equation}\label{eq:1xy}
e^{\frac{2 \pi \ii k}{n}}\cdot [1:xe^{\ii t^\prime}:y]=[1:xe^{\ii t}:y].
\end{equation}
The choice $s=-\frac{2\pi k}{n}-\sigma \mod 2\pi$, where $0 \leq s \leq 2\pi$ with the 
help~\eqref{eq:sigmaact} and~\eqref{eq:1xy}
we have that
\begin{equation}
\begin{split}
e^{-\ii s}\cdot [1:b:c]&=(e^{\frac{2 \pi \ii k}{n}} e^{\ii\sigma})\cdot [1:b:c]= e^{\frac{2 \pi \ii k}{n}} \cdot \left( e^{\ii \sigma} \cdot [1:b:c]\right)\\
&=e^{\frac{2 \pi \ii k}{n}} \cdot 
[1:xe^{\ii t^\prime}:y]=[1:xe^{\ii t}:y].
\end{split}
\end{equation}
The preceding equality with the help 
of~\eqref{eq:action} yields
\[
[1:b:c]=e^{\ii s} \cdot [1:xe^{\ii t}:y]=[1:xe^{\ii t} e^{-(n-m)s}:
ye^{-\ii ns}]\quad \textrm{ with }\quad  
|t|\leq \frac{\pi}{n}.
\]
By the definition of $\Delta_0$ and $\Gamma_0$ given in~\eqref{eq:defgammadelta} we note that they are disjoint. Moreover,  
by Remark~\ref{rem:1} we have that 
$\Pi_n(\Omega_n)=
\Delta_0\cup \Gamma_0$.

We point out that for $[1:x:y] \in \Gamma_0$ there is nothing to prove.
In the sequel, we assume that $[1:x:y] \in \Delta_0$. 
For short, we write $\omega=\omega(x,y)$.
We claim that 
\begin{equation}
|t|<\frac{(2\omega -n+1)\pi}{n}.
\end{equation}
Indeed, by contradiction argument assume that 
\begin{equation}\label{eq:contra}
\frac{(2\omega -n+1)\pi}{n}\leq|t|\leq \frac{\pi}{n}.
\end{equation}
Recall that $[1:b:c]\in \Omega_n$ and hence 
by~\eqref{eq:igualdad} we also have $[1:x:y]\in \Omega_n$.
Now, we apply Theorem~\ref{th:bohl} for the corresponding  trinomials $f$ and $g$ associated to 
$[1:x:y]$ and $[1:b:c]=[1: x e^{\ii t} e^{-\ii(n-m)s}: ye^{-\ii ns}]$, respectively. 
Since $[1,x,y]\in \Delta_0$, Remark~\ref{rem:2} yields  that the positive numbers $1$, $x$ and $y$ are the side lengths of some triangle (it may be degenerate).
The corresponding pivots~\eqref{def:PW} are given by
\begin{equation}\label{eq:pivotales}
P_f:=\frac{n-m}{2},\quad \quad
P_g:=\frac{n(t+\pi)-m\pi}{2\pi}=P_f+\frac{nt}{2\pi}
\end{equation}
and
\[
\omega_f(1)=
\omega_g(1)=\frac{n\omega_1+m\omega_2}{2\pi},
\]
where $\omega_1$ and $\omega_2$ are the opposite angles to the side lengths $1$ and $x$ of a triangle with side lengths $1$, $x$
and $y$.
Recall that $|t|\leq \frac{\pi}{n}$.
By~\eqref{eq:contra} and~\eqref{eq:pivotales}
\[
\frac{2\omega-(n-1)}{2}\leq |P_f-P_g|=\frac{n|t|}{2\pi}\leq \frac{1}{2}.
\] 
The preceding inequality with the help of Item~(2) and Item~(3) of Lemma~\ref{formareducida} yields that the trinomial $g\not \in \Omega_n$, which is a contradiction.
This finishes the proof of the necessity implication.

We continue with the proof of the sufficient implication.
Assume that  
\[[1:xe^{\ii t}e^{-\ii(n-m)s}: ye^{-\ii n s}]
\] satisfies (1) and (2) of Theorem~\ref{th:main}.
We show that $[1:xe^{\ii t}e^{-\ii(n-m)s}: ye^{-\ii n s}]\in \Omega_n$.

We recall that $\Delta_0$ and $\Gamma_0$ are disjoint and
$\Pi_n(\Omega_n)=
\Delta_0\cup \Gamma_0$.
For $[1:x:y] \in \Gamma_0$,~\eqref{eq:bolhsegunda} in Theorem~\ref{th:bohl}
implies that $[1:x:y]\in \Omega _n$. 

In the sequel, we assume that $[1:x:y] \in \Delta_0$.
We apply Theorem~\ref{th:bohl} for the corresponding  trinomials $f$ and $g$ associated to 
$[1:x:y]$ and $[1:b:c]=[1: x e^{\ii t} e^{-\ii(n-m)s}: ye^{-\ii ns}]$, respectively. 
By~\eqref{eq:pivotales} we obtain $|P_f-P_g|=\frac{n|t|}{2\pi}$ and for $|t|<\frac{\pi (2 \omega-n+1)}{n}$ we have that
\[
|P_f-P_g|<
\frac{2\omega-(n-1)}{2}.
\]
The preceding inequality with the help of Item~(1) of Lemma~\ref{formareducida} implies that the trinomial $g \in \Omega_n$.
The proof of the necessity implication is complete.
\end{proof}

\appendix
\section{\textbf{Parity argument and projection to the real case}}\label{ap:tools}
This section contains useful properties that help us to make this manuscript more fluid. 
\begin{lemma}[Parity argument for special pivots]\label{parity}
\hfill

\noindent
For the open interval
\[
I=(P-\omega, P+\omega)
\] 
with $P\in \mathbb{R}$ and $\omega > 0$  the following statements holds true.
\begin{enumerate}
\item[(i)] If $P\in \mathbb{Z}$, then 
$\#I$ is an odd positive integer.

\item[(ii)] If  
$P\in \{\ell/2 :\ell
\textrm{ is an odd number}\}$, then
$\#I$ is an even non-negative integer.
\end{enumerate}
\end{lemma}

\begin{proof}
We start with the proof of Item~(i). We observe that $P+k\in (P,P+\omega)$ for some $k\in \mathbb{N}$ if and only if $P-k\in (P-\omega,P)$ for some $k\in \mathbb{N}$.
The previous observation with the help of the hypothesis $P\in \mathbb{Z}$ implies the statement. 
The proof of Item~(ii) is analogous.
\end{proof}

\begin{remark}
In the following lemma, for convenience in the proof we introduce the following redundant inequality $2\omega \leq n$ 
in the definition of $\Delta _j$, $j=0,1$, as one can see along its proof.
\end{remark} 

\begin{lemma}[Projection to the real case]\label{lem:piezas}
\hfill

\noindent
Let $\Pi_n: \Omega_n \to \mathbb{P}^2 _{\mathbb{R}}$ be defined by
\begin{equation}\label{def:pi2} 
\Pi_n([1:b:c])=[1:|b|:(-1)^n|c|],
\end{equation}
where $\Omega_n$ is given in~\eqref{def:omega}.
Then it follows that 
\begin{equation}\label{eq:pnomegan}
\Pi_n(\Omega_n)=
\begin{cases}
\Delta_0\cup \Gamma_0 & \textrm{ for $n$ being an even number},\\
\Delta_1\cup \Gamma_1 & \textrm{ for $n$ being an odd number},
\end{cases}
\end{equation}
where 
\begin{equation}
\begin{split}
\Gamma_{0}&=\{[1:x:y] \in \mathbb{P}_{\mathbb{R}}^2 \, \,  
| \, \,  0 \leq x,\, \,  0 \leq y, \, \,  x+y<1\},\\
\Delta_{0}&= \{[1:x:y] \in \mathbb{P}_{\mathbb{R}}^2 \, \,  | \, \,  0 < x,\, \,  0 < y, \, \, 1 \leq x+y  , \, \, n-1<2\omega \leq n \},\\
\Gamma _{1}&=\{[1:x:y] \in \mathbb{P}_{\mathbb{R}}^2 \, \,  
| \, \,  0 \leq x,\, \,  y \leq 0, \, \,  x-y<1\}, \\
\Delta_{1}&=\{[1:x:y] \in \mathbb{P}_{\mathbb{R}}^2 \, \,  | \, \,  0 < x,\, \,  y < 0, \, \, 1 \leq x-y , \, \, n-1<2\omega \leq n\}, 
\end{split}
\end{equation} 
and
\[
\omega:= \frac{n\, \arccos \left(\frac{x^2+y^2-1}{2x|y|}\right)+m\, \arccos \left( \frac{1-x^2+y^2}{2|y|}\right)}{2\pi}.\]
\end{lemma}
\begin{proof}
We show the  case when $n$ is an even number. The proof for the case when $n$ is an odd number is analogous and we omit it. 

We note that  $\Pi_n([1:b:c])=[1:|b|:|c|]$. Proposition~7.9 in~\cite{BarreraMaganaNavarrete} yields
\[
\Pi_n(\Omega _n)\subset \{ [1:x:y] \in \mathbb{P} ^2 _{\mathbb{R}} : 0\leq x,\, 0\leq y \} \cap \Omega _n.
\]
Moreover, since $\Pi_n$ restricted to the set $\{ [1:x:y] \in \mathbb{P} ^2 _{\mathbb{R}} : 0\leq x,\, 0\leq y \} \cap \Omega _n$ is the identity map, we have that
\[\Pi_n(\Omega _n)= \{ [1:x:y] \in \mathbb{P} ^2 _{\mathbb{R}} : 0\leq x,\, 0\leq y \} \cap \Omega _n.
\]
Now, we show 
$[1:x:y]\in \Pi_n(\Omega_n)$ and $1$, $x$ and $y$ are the sides lengths of some triangle (it may be degenerate)
if and only if $[1:x:y]\in \Delta_0$.

Let $[1:x:y]\in \Pi_n(\Omega_n)$
and assume that there exists a triangle with side lengths $1,x,y$ then
$$0<x,\quad 0<y\quad \textrm{ and }\quad 1\leq x+y.$$
Theorem~\ref{th:bohl} with the help of the Law of Cosines implies that $\# I=n$, where $I=(P-\omega, P+\omega)$, $P=\frac{n-m}{2}$ and 
$$\omega=\frac{n \arccos \left(\frac{x^2+y^2-1}{2xy}\right)+m \arccos \left( \frac{1-x^2+y^2}{2y}\right)}{2\pi}.$$
By Lemma~\ref{basicfactsonintervals1} we have that  $n-1< 2\omega<n+1$. 
We claim that  $n-1< 2\omega \leq n$. Otherwise, Theorem~\ref{th:bohl}, 
Lemma~\ref{basicfactsonintervals2} Item~(iii) and Lemma~\ref{formareducida} Item~(3) imply the existence of a trinomial equation of degree $n$
associated to $[1:xe^{\ii t}:y]$ with $ \frac{\pi}{n}(n+1-2\omega)<|t|\leq \frac{\pi}{n}$
having $n+1$ roots, which is an absurd. 
Therefore, $[1:x:y] \in \Delta _0$.

Conversely, assume that $[1:x:y] \in \Delta _0$. By Remark~\ref{rem:2}  there exists a triangle with side lengths equal to $1$, $x$ and $y$. Hence, Theorem~\ref{th:bohl}, Item~(ii) of Lemma~\ref{parity} in Appendix~\ref{ap:tools}, and Item~(i) and 
Item~(ii) of
Lemma~\ref{basicfactsonintervals2}  imply that 
\[
[1:x:y] \in \Omega _n \cap \{ [1:x:y] \in \mathbb{P} ^2 _{\mathbb{R}} : 0\leq x, 0\leq y \}.
\]

Finally, applying the third inequality 
of~\eqref{eq:bolhsegunda} in  Theorem~\ref{th:bohl} we have that 
$[1:x:y] \in \Pi_n(\Omega _n)$ and hence we deduce that there is no triangle with length sides $1$, $x$ and $y$ if and only if $[1:x:y] \in \Gamma_0$.
The proof is complete.
\end{proof}

\subsection{\textbf{Proof of Corollary~\ref{cor:stablereal}}}\label{sub:proofcorstable}
\hfill

\noindent
For $(n,m)$, with $n$ being an even positive integer.
Since $\gcd(n,m)=1$, we have that $m$ is an odd integer. By~\eqref{def:pi} we have that $x>0$ and $y>0$.
We observe the following
\begin{itemize}
\item[(1)] for any $[1:x:y]\in \Delta_0\cup \Gamma_0$ the choice $t=0$ and $s=\pi$ yields
\[[1:-x:y]\in 
\Omega_n \cap \mathbb{P}^2_{\mathbb{R}}.
\]
\item[(2)] for any $[1:x:y]\in \Gamma_0$ the choice $t=\pi/n$ and $s=k\pi/n$ with
$k$ being an odd integer on $0<k<2n$  such that 
$e^{\ii\pi/n}=e^{\ii(n-m)k\pi/n}$
 implies
\[[1:x:-y]\in 
\Omega_n \cap \mathbb{P}^2_{\mathbb{R}}.
\]
Such $k$ exists due to $\gcd(n-m,2n)=1$.
\item[(3)] By Item~(2) we know that $[1:x:-y]\in \Omega_n \cap \mathbb{P}^2_{\mathbb{R}}$ whenever $[1:x:y]\in \Gamma_0$.
We note that the choice $t=0$ and $s=\pi$ gives
\[[1:-x:-y]=[1:xe^{\ii t}e^{-\ii (n-m)s}:(-y)e^{-\ii ns}]
\in \Omega_n \cap \mathbb{P}^2_{\mathbb{R}}
\]
when $[1:x:y]\in \Gamma_0$.
\item[(4)] 
We now note that if $[1:x:y]\in \Delta_0$ then $[1:x:-y]\not\in \Omega_n \cap \mathbb{P}^2_{\mathbb{R}}$. 
Indeed, since $[1:x:y]\in \Delta_0$ we know that there exists a triangle (it may be degenerate) with side lengths $1$, $x=|x|>0$ and $y=|y|>0$.
Recall that $n$ is an even number. 
By~\eqref{eq:Pcasos} we have that corresponding pivot $P$ for the point $[1:x:-y]$ is an integer number and then Item~(i) in Lemma~\ref{parity} in Appendix~\ref{ap:tools} yields that any non-empty open interval centred at $P$ contains an odd number of integers. Hence, $[1:x:-y]\not\in \Omega_n \cap \mathbb{P}^2_{\mathbb{R}}$.
Similarly, we deduce that  $[1:-x:-y]\not\in \Omega_n \cap \mathbb{P}^2_{\mathbb{R}}$ whenever  
$[1:x:y]\in \Delta_0$.
\end{itemize}

Now, we consider $n$ be an odd positive integer. In the sequel,  we assume that $m$ is an even number and recall that 
$\gcd(n,m)=1$. Then we have that $n-m$ is an odd integer.
By~\eqref{def:pi} we have that $x>0$ and $y<0$.
We observe the following:
\begin{itemize}
\item[(i)] For any $[1:x:y]\in \Delta_1\cup \Gamma_1$ the choice $t=0$ and $s=\pi$ yields
\[[1:-x:-y]\in 
\Omega_n \cap \mathbb{P}^2_{\mathbb{R}}.
\]
\item[(ii)] For any $[1:x:y]\in \Gamma_1$ the choice $t=\pi/n$ and $s=k\pi/n$ with
$k$ being an odd integer on $0<k<2n$  such that 
$e^{\ii\pi/n}=
e^{\ii (n-m)k\pi/n}$
 implies
\[[1:x:-y]\in 
\Omega_n \cap \mathbb{P}^2_{\mathbb{R}}.
\]
Such $k$ exists due to $\gcd(n-m,2n)=1$.
\item[(iii)] By Item~(ii) we know that $[1:x:-y]\in \Omega_n \cap \mathbb{P}^2_{\mathbb{R}}$ whenever $[1:x:y]\in \Gamma_1$.
We note that the choice $t=0$ and $s=\pi$ gives
\[[1:-x:y]=[1:xe^{\ii t}e^{-\ii (n-m)s}:(-y)e^{-\ii ns}]
\in \Omega_n \cap \mathbb{P}^2_{\mathbb{R}}.
\]
\item[(iv)] 
We now note that if $[1:x:y]\in \Delta_1$ then $[1:x:-y]\not\in \Omega_n \cap \mathbb{P}^2_{\mathbb{R}}$. 
Indeed, since $[1:x:y]\in \Delta_1$ we know that there exists a triangle (it may be degenerate) with side lengths $1$, $x=|x|>0$ and $-y=|y|>0$.
Recall that $n$ is an odd integer, $m$ is an even integer and $(n-m)/2$ is not integer.
By~\eqref{eq:Pcasos} we have that corresponding pivot $P$ for the point $[1:x:-y]$ is $(n-m)/2$ and then Item~(ii) in Lemma~\ref{parity} in Appendix~\ref{ap:tools} yields that any non-empty open interval centred at $P$ contains an even number of integers. Hence, $[1:x:-y]\not\in \Omega_n \cap \mathbb{P}^2_{\mathbb{R}}$.
Similarly, we deduce that  $[1:-x:y]\not\in \Omega_n \cap \mathbb{P}^2_{\mathbb{R}}$ whenever  
$[1:x:y]\in \Delta_1$.
\end{itemize}

Finally, we consider $n$ be an odd positive integer. In the sequel,  we assume that $m$ is an odd number and recall that 
$\gcd(n,m)=1$. Then we have that $n-m$ is an even integer.
By~\eqref{def:pi} we have that $x>0$ and $y<0$.
We observe the following:
\begin{itemize}
\item[(a)] For any $[1:x:y]\in \Delta_1\cup \Gamma_1$ the choice $t=0$ and $s=\pi$ yields
\[[1:x:-y]\in 
\Omega_n \cap \mathbb{P}^2_{\mathbb{R}}.
\]
\item[(b)] By the choice  $t=0$ and $s=\pi$ we note   that $[1:-x:y]\in \Omega_n \cap \mathbb{P}^2_{\mathbb{R}}$ if and only if $[1:-x:-y]\in \Omega_n \cap \mathbb{P}^2_{\mathbb{R}}$.
\item[(c)] For any $[1:x:y]\in \Gamma_1$ the choice $t=\pi/n$ and $s=k\pi/n$ with
$k$ being an integer on $0<k<n$  such that 
$-e^{\ii\pi/n}=e^{\ii (n+1)\pi/n}=e^{\ii (n-m)k\pi/ n}$
 implies
\[[1:-x:(-1)^k y]\in 
\Omega_n \cap \mathbb{P}^2_{\mathbb{R}}.
\]
Such $k$ exists due to $\gcd(n-m,n)=1$. By Item~(b) we conclude that
\[[1:-x:\pm y]\in 
\Omega_n \cap \mathbb{P}^2_{\mathbb{R}}.
\]
\item[(d)] 
We now note that if $[1:x:y]\in \Delta_1$ then $[1:-x:y]\not\in \Omega_n \cap \mathbb{P}^2_{\mathbb{R}}$. 
Indeed, since $[1:x:y]\in \Delta_1$ we know that there exists a triangle (it may be degenerate) with side lengths $1$, $x=|x|>0$ and $-y=|y|>0$.
Recall that $n$ and $m$ are odd integers and $n-m$ is an even integer.  
By~\eqref{eq:Pcasos} we have that corresponding pivot $P$ for the point $[1:-x:y]$ is $n/2$ and then Item~(ii) in Lemma~\ref{parity} in Appendix~\ref{ap:tools} yields that any non-empty open interval centred at $P$ contains an even number of integers. Hence, $[1:-x:y]\not\in \Omega_n \cap \mathbb{P}^2_{\mathbb{R}}$.
Similarly, we deduce that  $[1:-x:-y]\not\in \Omega_n \cap \mathbb{P}^2_{\mathbb{R}}$ whenever  
$[1:x:y]\in \Delta_1$.
\end{itemize}

\section{\textbf{Bohl's Theorem}}
\label{ap:Bohlstatement}

In this section, 
we stress the equivalence, under the assumption that $n$ and $m$ are coprime numbers, of 
Theorem~\ref{th:bohl} with the original statement of Bohl's Theorem~\cite{Bohl}, which for completeness of the presentation we state it below with its original notation.
\begin{theorem}[Bohl's Theorem \cite{Bohl}]\label{th:Bohloriginal}
\hfill

\noindent
The number of roots of \eqref{def:f}, 
that have modulus strictly smaller than the positive number $r$, are obtained by taking the $\tau$-multiple of a number $\zeta$, where $\tau = \textsf{gcd}(n,m)$ and $\zeta$ can be determined as follows:
\begin{enumerate}
\item[I.] If any of the quantities $|a|r^n$, $|b|r^m$, $|c|$ is strictly smaller than the sum of the other two, then they constitute a triangle, such that the sides are proportional to the preceding quantities.
Let 
$\omega_1$ and $\omega_2$ denote the two angles which are opposite to $|a|r^n$ and $|b|r^m$, respectively. Then $\zeta$ is given by the number of integers, which lie between
\begin{equation}
\label{Pmenosomega}
\frac{n(\beta-\gamma+\pi)-m(\alpha-\gamma+\pi)}{2\tau\pi}-\frac{n\omega_1+m\omega_2}{2\tau\pi}
\end{equation}
and
\begin{equation}\label{Pmasomega}
\frac{n(\beta-\gamma+\pi)-m(\alpha-\gamma+\pi)}{2\tau\pi}+\frac{n\omega_1+m\omega_2}{2\tau\pi},
\end{equation}
where $\alpha$, $\beta$, $\gamma$ are the arguments of $a$, $b$, $c$, respectively.
\item[II.] One of the three quantities $|a|r^n$, $|b|r^m$ and $|c|$ is greater or equal to the sum of the other two. If we exclude the exceptional cases (a) and (b) mentioned below, we have that\\
\begin{itemize}
\item[II.1] if $|c| \geq |a|r^n+|b|r^m$, then $\zeta=0$,
\item[II.2] if $|b|r^m \geq |a|r^n+|c|$, then $\zeta=\frac{m}{\tau}$,
\item[II.3] if $|a|r^n \geq |b|r^m+|c|$, then $\zeta=\frac{n}{\tau}$.
\end{itemize}
The following exceptional cases occur:\\
\begin{itemize}
\item[(a)] If $|b|r^m=|a|r^n+|c|$ and $r^{n-m}\leq \frac{m |b|}{n|a|}$ hold true
and additionally
\begin{equation}\label{eq:entero}
\frac{1}{\tau}\left( \frac{n(\beta -\gamma)-m(\alpha -\gamma)}{\pi}+n\right)
\quad \textrm{ is an even integer},
\end{equation}
then  $\zeta=\frac{m}{\tau}-1$.
\item[(b)] If $|a|r^n=|b|r^m+|c|$ holds true and additionally
\begin{equation}\label{eq:enterodos}
\frac{1}{\tau}\left( \frac{n(\beta -\gamma)-m(\alpha -\gamma)}{\pi}-m\right)
\quad \textrm{ is an even integer},
\end{equation}
then $\zeta=\frac{n}{\tau}-1$.
\end{itemize}
\end{enumerate}
\end{theorem}

In the sequel, we stress that the counting procedure provided in Theorem~\ref{th:bohl} has already taken in account the exceptions~(a) and~(b) given in Theorem~\ref{th:Bohloriginal}.

We start noticing that since we assume $\tau=\textsf{gcd}(n,m)=1$, the numbers 
$P-\omega(r)$ and $P+\omega(r)$ defined in~\eqref{def:PW} are equal to~\eqref{Pmenosomega} and~\eqref{Pmasomega}, respectively.

Now, we discuss the exceptional case~(a).
By Descartes' rule of signs, there are at most two distinct positive real roots of the equation 
$|b|x^m=|a|x^n+|c|$. We denote by $r_1$ and $r_2$ such roots and without loss of generality we assume that $0<r_1\leq r_2$.
Moreover, $r_1^{n-m} \leq \frac{m |b|}{n|a|}\leq r_2 ^{n-m}$, see~\cite{Bohl}, pp.~560--561.

Suppose that $r_1^{n-m} \leq \frac{m |b|}{n|a|}$. Since $|b|r^m_1=|a|r^n_1+|c|$,
there is a degenerate triangle with side lengths $|a|r_1^n$, $|b|r_1^m$, $|c|$ and the
angles opposite  to $|a|r_1^n$ and $|b|r_1^m$ are given by
$\omega _1= 0$ and $\omega _2=\pi$, respectively. Hence,
$\omega(r_1)=\frac{m}{2}$ and the length of the interval $(P-\omega(r_1), P+\omega(r_1))$ is equal to $2\omega(r_1)=m$. By~\eqref{eq:entero}, we have that
\[
P+\omega(r_1)=\frac{1}{2}\left( \frac{n(\beta -\gamma)-m(\alpha -\gamma)}{\pi}+n\right)\quad \textrm{is an integer},
\]
which together with $2\omega(r_1)=m$ implies that
$P-\omega(r_1)$ is also an integer. Therefore, the number of integers contained in the interval $(P-\omega(r_1), P+\omega(r_1))$ is equal to
$m-1$. As a consequence, when $r_1^{n-m} \leq \frac{m |b|}{n|a|}$, the counting procedure given in Theorem~\ref{th:bohl} agrees with the conclusion of Item~(a).

Now, we remark that when $\frac{m |b|}{n|a|}<r_2 ^{n-m}$ and $P+\omega(r_2)$ is an integer, 
Theorem~\ref{th:bohl} and Theorem~\ref{th:Bohloriginal} state that $\zeta=m$.
 
In the sequel, we analyze the exceptional case~(b).
Analogously to the case~(a), there is a unique positive root of the equation $|a|x^n=|b|x^m+|c|$, and we denote such root by $r_0$. Hence, there is a degenerate triangle with side lengths $|a|r_0^n$, $|b|r_0^m$, $|c|$ and the
angles opposite  to $|a|r_0^n$ and $|b|r_0^m$ are given by
$\omega _1= \pi$ and $\omega _2=0$, respectively. Hence,
$\omega(r_0)=\frac{n}{2}$ and the length of the interval $(P-\omega(r_0), P+\omega(r_0))$ is equal to $2\omega(r_0)=n$. By \eqref{eq:enterodos} we have
\[
P-\omega(r_0)=\frac{1}{2}\left( \frac{n(\beta -\gamma)-m(\alpha -\gamma)}{\pi}-m\right)\quad
\textrm{is an integer},
\] 
which yields that $P+\omega(r_0)$ is an integer. Therefore, the number of
integers contained in the interval $(P-\omega(r_0), P+\omega(r_0))$ is equal to
$n-1$.
In summary, the statement of Item~(b) of Theorem~\ref{th:Bohloriginal} agrees with the counting given in Theorem~\ref{th:bohl}.

\section*{\textbf{Declarations}}
\noindent
\textbf{Acknowledgments.} 
G. Barrera would like to express his gratitude to University of Helsinki, Department of Mathematics and
Statistics, for all the facilities used along the realization of this work. 
He thanks the Faculty of Mathematics, UADY, Mexico, for the hospitality during the research visit in 2022, where partial work on this paper was undertaken. He would also like to thank the Instituto de Matem\'atica Pura e Aplicada (IMPA), Brazil, for support and hospitality during the 2023 Post-Doctoral Summer Program, where partial work on this paper was undertaken.
All authors are grateful for the invitation to the Banff BIRS-CMO online scientific activity ``Real Polynomials: Counting and Stability'' (2021), which motivates this paper.
The authors would like to thank prof. Jonas M. T\"olle (Department of Mathematics and Systems Analysis, Aalto University, Espoo, Finland) for his support on the translation of~\cite{Bohl}. The authors are grateful to the reviewer for the thorough examination of the paper, which has lead to a significant improvement.
Figure~\ref{triangulo} and Figure~\ref{pindeomegan} were created using the open source software GeoGebra.

\noindent
\textbf{Ethical approval.} Not applicable.

\noindent
\textbf{Competing interests.} The authors declare that they have no conflict of interest.

\noindent
\textbf{Authors' contributions.}
All authors have contributed equally to the paper.

\noindent
\textbf{Funding.} 
The research of G. Barrera has been supported by the Academy of Finland, via the Matter and Materials Profi4 University Profiling Action, an Academy project (project No. 339228) and the Finnish Centre of Excellence in Randomness and STructures (project No. 346306). 
The research of W. Barrera and J. P. Navarrete has been supported by the CONACYT, ``Proyecto Ciencia de Frontera'' 2019--21100 via Faculty of Mathematics, UADY, M\'exico.

\noindent
\textbf{Availability of data and materials.} Not applicable.

\bibliographystyle{amsplain}

\end{document}